%% file: polsys.tex
\documentclass[graybox]{svmult}

\usepackage{type1cm}        
\usepackage{makeidx}         
\usepackage{graphicx}        
\usepackage{multicol}        
\usepackage[bottom]{footmisc}

\usepackage{newtxtext}       %
\usepackage[varvw]{newtxmath}       

\makeindex             

\usepackage{xcolor}
\usepackage{graphicx}
\usepackage{tikz,pgfplots}
\usepackage{chemformula,chemfig,chemmacros}
\usepackage{kbordermatrix}
\usepackage{comment}
\usepackage{subcaption}
\usepackage{enumitem}

\definecolor{mycolor1}{rgb}{0.00000,0.44700,0.74100}
\definecolor{mycolor2}{rgb}{0.8500, 0.3250, 0.0980}
\definecolor{mycolor3}{rgb}{0.9290, 0.6940, 0.1250}
\definecolor{mycolor4}{rgb}{0.4940, 0.1840, 0.5560}
\definecolor{mycolor5}{rgb}{0.4660, 0.6740, 0.1880}

\newcommand{\Kbar}{\overline{K}}

\begin{document}

\title*{Polynomial Equations: Theory and Practice}
\author{Simon Telen}
\institute{Simon Telen \at CWI, Science Park 123, 1098XG Amsterdam and Max Planck Institute for Mathematics in the Sciences (current), Inselstrasse 22, 04103 Leipzig, \email{simon.telen@mis.mpg.de}}
%
%
\maketitle

\abstract*{Solving polynomial equations is a subtask of polynomial optimization. This article introduces systems of such equations and the main approaches for~solving them. We discuss critical point equations, algebraic varieties, and solution counts. The theory is illustrated by many examples using different software packages.}

\abstract{Solving polynomial equations is a subtask of polynomial optimization. This article introduces systems of such equations and the main approaches for~solving them. We discuss critical point equations, algebraic varieties, and solution counts. The theory is illustrated by many examples using different software packages.}

\section{Polynomial equations in optimization}
\label{sec:1}

Polynomial equations appear in many fields of science and engineering. Some examples are chemistry \cite{dickenstein2016biochemical,muller2016sign}, molecular biology \cite{emiris1999computer}, computer vision \cite{kukelova2008automatic}, economics and game theory \cite[Chapter~6]{sturmfels2002solving}, topological data analysis \cite{breiding2020algebraic}, and partial differential equations \cite[Chapter~10]{sturmfels2002solving}. For an overview and more references, see \cite{breiding2021nonlinear,cox2020applications}. This article will be a chapter in the forthcoming book \emph{Polynomial optimisation, moments and applications} presenting research acitivies conducted in the European Network POEMA. In that context, polynomial equations arise from optimization problems.

Let us consider the problem of minimizing a polynomial objective function $f(x_1, \ldots, x_k) $ over the set $X = \{ x \in \mathbb{R}^k \, :\, h_1(x) = \cdots = h_\ell(x) = 0 \} \subset \mathbb{R}^k$, where also $h_1, \ldots, h_\ell$ are polynomials in the $k$ variables $x=(x_1, \ldots, x_k)$. This is a \emph{polynomial optimization problem} \cite{laurent2009sums}, often written as
\begin{equation} \label{eq:opt}
\setlength{\jot}{3pt}
\begin{aligned}
&\!\min_{x \in \mathbb{R}^k}        &\qquad& f(x_1, \ldots, x_k), \\
&\text{subject to} &      & h_1(x_1, \ldots, x_k) \, = \,  \cdots \,  = \, h_\ell(x_1, \ldots, x_k) \, = \,  0. 
\end{aligned}
\end{equation} 
Introducing new variables $\lambda_1, \ldots, \lambda_\ell$ we obtain the Lagrangian $L = f - \lambda_1h_1 - \cdots - \lambda_\ell h_\ell$, whose partial derivatives give the optimality conditions
\begin{equation} \label{eq:lagopt}
\frac{\partial L}{\partial x_1}\, = \,  \cdots \, = \,  \frac{\partial L}{\partial x_k} \, = \,  h_1 \, = \,  \cdots \, = \,  h_\ell \, = \,  0.
\end{equation} 
A solution in $\mathbb{R}^k$ is a candidate minimizer. Many methods for solving the equations \eqref{eq:lagopt}, like those presented in Section \ref{sec:methods}, compute all complex solutions first and then select the real ones among them. The number of solutions over $\mathbb{C}$ is typically finite. We present two examples of \eqref{eq:opt}. First, we minimize the distance to algebraic varieties.
\begin{example}{Example: Euclidean distance degree}
Given a point $y = (y_1, \ldots, y_k) \in \mathbb{R}^k$, we consider the squared Euclidean distance function $f(x_1, \ldots, x_k) = \lVert x-y \rVert_2^2 = (x_1 - y_1)^2 + \cdots + (x_k - y_k)^2$. As above, $X$ is the set  
$\{ x \in \mathbb{R}^k ~|~ h_1 = \cdots = h_\ell = 0 \}$.
 The solution $x^*$ of the optimization problem \eqref{eq:opt} is
 \begin{equation} \label{eq:EDdeg}
 x^* \, = \, \arg \min_{x \in X} \, \lVert x - y \rVert^2_2, 
 \end{equation}
 i.e.~the point on $X$ that is closest to $y$. The algebraic complexity of this problem is studied in \cite{draisma2016euclidean}. 
 For instance, let $k = 2, \ell = 1$ and let $X$ be the unit ball with respect to the 4-norm $\lVert \cdot \rVert_4$: $h = h_1 = x_1^4 + x_2^4 - 1$. We want to find the point on $X$ closest to $y = (2, 1.4)$. In \texttt{Mathematica} \cite{mathematica}, one solves \eqref{eq:lagopt} as follows:
 
 \scriptsize
 \begin{verbatim}
      f = (x1 - 2)^2 + (x2 - 1.4)^2; h = x1^4 + x2^4 - 1; L = f - lambda*h;
      NSolve[{D[L, x1] == 0 &&  D[L, x2] == 0 &&  h == 0}, Reals]
 \end{verbatim}
 \normalsize  
 
 \noindent This returns two critical points on $X$. One of them minimizes the distance to $y$, the other maximizes it. The minimizer is $x^*= (0.904944, 0.757564)$. If we delete the option \texttt{Reals}, the program returns 16 complex solutions.
\end{example}
\noindent Second, we set up a polynomial optimization problem from system identification. 
\begin{example}{Example: parameter estimation for system identification}
System identification is an engineering discipline that aims to construct models for dynamical systems from measured data. A model explains the relationship between input, output, and noise. It depends on a set of \emph{model parameters}, which are selected to best fit the measured data. A \emph{discrete-time, single-input single-output linear time-invariant system} with input sequence $u: \mathbb{Z} \rightarrow \mathbb{R}$, output sequence $y: \mathbb{Z} \rightarrow \mathbb{R}$ and white noise sequence $e: \mathbb{Z} \rightarrow \mathbb{R}$ is often modeled by
\begin{equation} \label{eq:SISO}
A(q) \, y(t) \, = \,  \frac{B_1(q)}{B_2(q)} \, u(t) + \frac{C_1(q)}{C_2(q)} \, e(t).
\end{equation}
Here $A, B_1, B_2, C_1, C_2 \in \mathbb{C}[q]$ are unknown polynomials of a fixed degree in the \emph{backward shift operator} $q$, acting on $s: \mathbb{Z} \rightarrow \mathbb{R}$ by $qs(t) = s(t-1)$. The model parameters are the coefficients of these polynomials, which are to be estimated. Clearing denominators in \eqref{eq:SISO} gives 
\begin{equation} \label{eq:pme}
A(q)B_2(q)C_2(q) y(t) = B_1(q)C_2(q) u(t) + B_2(q)C_1(q) e(t).
\end{equation}
Suppose we have measured $u(0), \ldots, u(N), y(0), \ldots, y(N)$. Let $d = \max(d_A + d_{B_2} + d_{C_2}, d_{B_1} + d_{C_2}, d_{B_2} + d_{C_1})$, where $d_A, d_{B_1}, d_{B_2}, d_{C_1}, d_{C_2}$ are the degrees of our polynomials. Writing \eqref{eq:pme} for $t = d, d+1, \ldots, N$, we find algebraic relations among the coefficients of $A,B_1,B_2,C_1,C_2$. The model parameters are estimated by solving
\begin{equation*} 
\setlength{\jot}{3pt}
\begin{aligned}
\min_{\Theta \in \mathbb{R}^k}        \qquad e(0)^2 + \ldots + e(N)^2 \quad \textup{subject to}       \quad \textup{\eqref{eq:pme} is satisfied for } t = d, \ldots, N 
\end{aligned}
\end{equation*} 
where $\Theta$ consists of $e(0), \ldots, e(N)$ and the coefficients of $A, B_1, B_2, C_1, C_2$. We refer to \cite[Section 1.1.1]{batselier2013numerical} for a worked-out example and more references. 
\end{example}
More general versions of \eqref{eq:opt} add inequality constraints of the type $q_1(x) \geq 0, \ldots, q_{\ell'}(x) \geq 0$, where $q_1, \ldots, q_{\ell'}$ are polynomials. Such problems can be handled using sums-of-squares relaxations \cite{lasserre2001global}. 

Our aim in this article is to introduce systems of polynomial equations in general, and methods for solving them. The reader is encouraged to try out these methods for polynomial optimization, for instance, in a Euclidean distance computation \eqref{eq:EDdeg} or in system identification. Section \ref{sec:varieties} discusses solution sets to polynomial equations, also called \emph{algebraic varieties}, and root finding over different fields. In Section \ref{sec:nosolutions}, we present several classical upper bounds on the number of solutions. Section \ref{sec:methods} is about normal form methods and homotopy continuation methods, which are two different important approaches to solving polynomial equations. Finally, Section \ref{sec:lines} contains a case study in which we apply these methods to compute 27 lines on a cubic surface. 

\textbf{Acknowledgements.} This article is based on an introductory lecture given at the workshop \emph{Solving polynomial equations and applications} organized at CWI, Amsterdam in October 2022. I thank Monique Laurent for involving me in this~workshop, and all other speakers and attendants for making it a success. 
I was supported by a Veni grant from the Netherlands Organisation for Scientific Research (NWO).

\section{Systems of equations and algebraic varieties} \label{sec:varieties}
Let $K$ be a field with algebraic closure $\Kbar$, e.g., $K = \mathbb{R}$ and $\Kbar = \mathbb{C}$. The polynomial ring with $n$ variables and coefficients in $K$ is $R = K[x_1, \ldots, x_n]$. We abbreviate $x = (x_1, \ldots, x_n)$ and use variable names $x,y,z$ rather than $x_1,x_2,x_3$ when $n$ is small. Elements of $R$ are polynomials, which are functions $f: \Kbar^n \rightarrow \Kbar$ of the form
\[ f(x) \, = \, \sum_{\alpha \in \mathbb{N}^n} c_{(\alpha_1, \ldots, \alpha_n)} \, x_1^{\alpha_1}\cdots x_n^{\alpha_n} \, = \,  \sum_{\alpha \in \mathbb{N}^n} c_\alpha \, x^\alpha, \]
with finitely many nonzero coefficients $c_\alpha \in K$.
A system of polynomial equations~is 
\begin{equation} \label{eq:polsys}
    f_1(x) \, = \, \cdots \, = \, f_s(x) \, = \, 0,
\end{equation}
where $f_1, \ldots, f_s \in R$. By a \emph{solution} of \eqref{eq:polsys}, we mean a point $x \in \Kbar^n$ satisfying all of these $s$ equations. \emph{Solving} usually means finding coordinates for all solutions. This makes sense only when the set of solutions is finite, which typically happens when $s \geq n$. However, systems with infinitely many solutions can be `solved' too, in an appropriate sense \cite{sommese2001numerical}. We point out that one is often mostly interested in solutions $x \in K^n$ over the ground field $K$. The reason for allowing solutions over the algebraic closure $\Kbar$ is that many solution methods, like those discussed in Section \ref{sec:methods}, intrinsically compute all such solutions. For instance, \eqref{eq:lagopt} is a polynomial system with $n = s = k + \ell$, and the field is $K = \mathbb{R}$. Here are some other examples.

\begin{example}{Example: univariate polynomials ($n=1$) and linear equations} 
When $n = s = 1$, solving the polynomial system defined by $f = a_0 + a_1 x + \cdots + a_d x^d \in K[x]$, with $a_d \neq 0$, amounts to finding the roots of $f(x) = 0$ in $\Kbar$. These are the eigenvalues of the $d \times d$ \emph{companion matrix} 
\begin{equation} \label{eq:companionmatrix}
    C_f \, = \, \begin{pmatrix}
     &  & &-a_0/a_d \\
     1 &  & & -a_1/a_d \\
     & \ddots & &\vdots \\
     & & 1 & -a_{d-1}/a_d
    \end{pmatrix}
\end{equation}
of $f$, whose characteristic polynomial is $\det(x \cdot {\rm id} - C_f) = a_d^{-1} \cdot f$.

When $f_i = \sum_{j=1}^n a_{ij} \, x_j - b_i$ are given by affine-linear functions, \eqref{eq:polsys} is a linear system of the form $Ax = b$, with $A \in K^{s \times n}$, $b \in K^s$. 
\end{example}
\noindent This example shows that, after a trivial rewriting step, the univariate and affine-linear cases are reduced to a \emph{linear algebra} problem. Here, we are mainly interested in the case where $n > 1$, and some equations are of degree $>1$. Such systems require tools from \emph{nonlinear algebra} \cite{michalek2021invitation}. We proceed with an example in two dimensions. 
\begin{example}{Example: intersecting two curves in the plane} 
Let $K = \mathbb{Q}$ and $n = s = 2$. We work in the ring $R = \mathbb{Q}[x,y]$ and consider the system of equations $f(x,y) = g(x,y) = 0$ where 
\begin{align} \label{eq:running}
\begin{split}
    f \, &= \, -7 x -9 y -10 x^2 + 17 xy + 10 y^2 + 16 x^2y -17xy^2, \\
    g \, &= \, 2 x -5 y +5 x^2 + 5 xy + 5 y^2 - 6 x^2y -6 xy^2.
    \end{split}
\end{align}
Geometrically, we can think of $f(x,y) = 0$ as defining a curve in the plane. This is the orange curve shown in Fig.~\ref{fig:curvesa}. The curve defined by $g(x,y) = 0$ is shown in blue. The set of solutions of $f = g = 0$ consists of points $(x,y) \in \overline{\mathbb{Q}}^2$ satisfying $f(x,y) = g(x,y) = 0$. These are the intersection points of the two curves. There are seven such points in $\overline{\mathbb{Q}}^2$, of which two lie in $\mathbb{Q}^2$. These are the points $(0,0)$ and $(1,1)$. Note that all seven solutions are real: replacing $\mathbb{Q}$ by $K = \mathbb{R}$, we count as many solutions over $K$ as over $\Kbar = \mathbb{C}$.
\begin{figure}
\begin{subfigure}[b]{0.45\textwidth}
    \centering
    \includegraphics[height = 5cm]{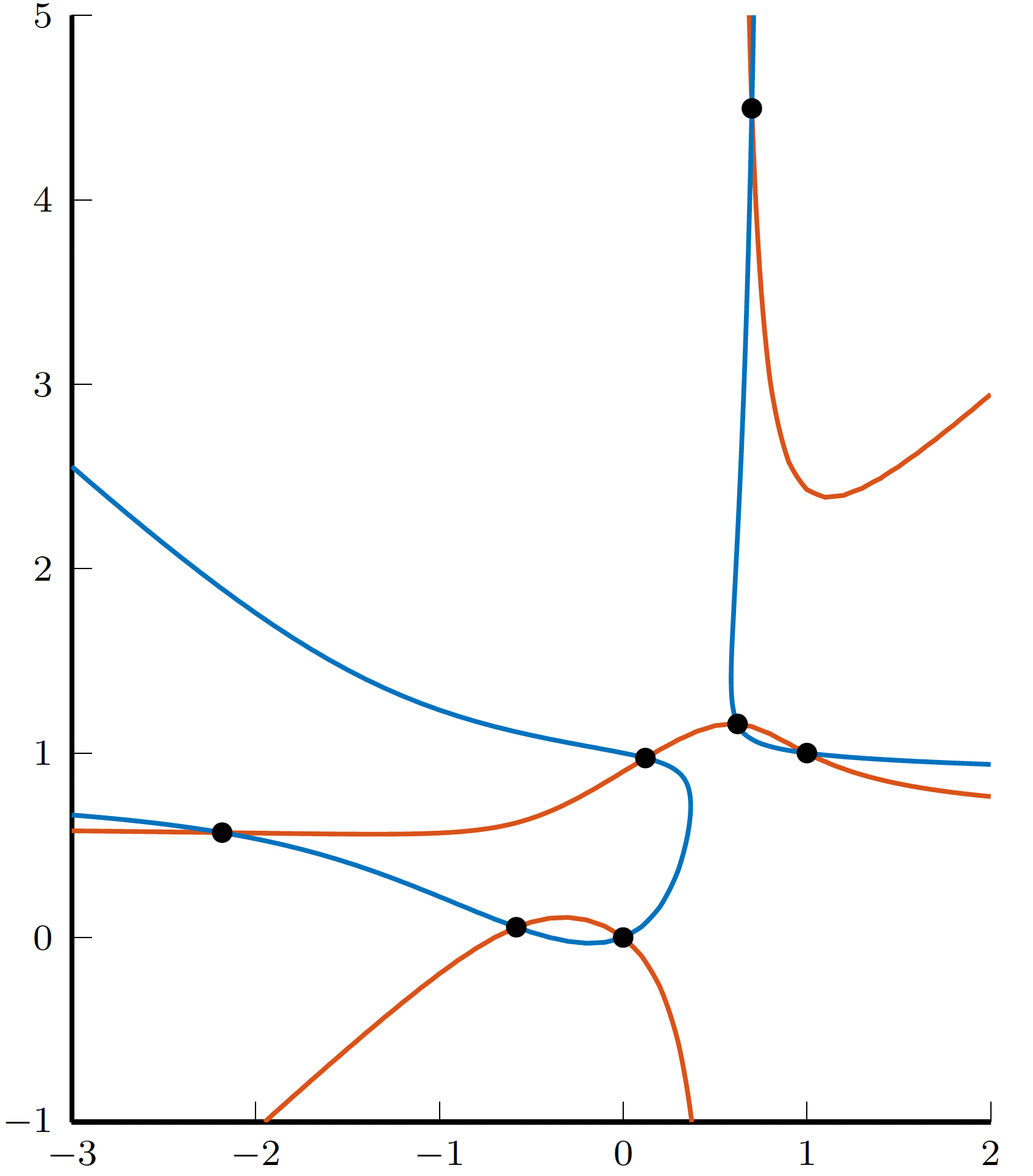} 
    \caption{Curves defined by $f,g$ from \eqref{eq:running}.}
    \label{fig:curvesa}
    \end{subfigure}
    \hfill
    \begin{subfigure}[b]{0.45\textwidth}
        \centering
        \includegraphics[height = 5cm]{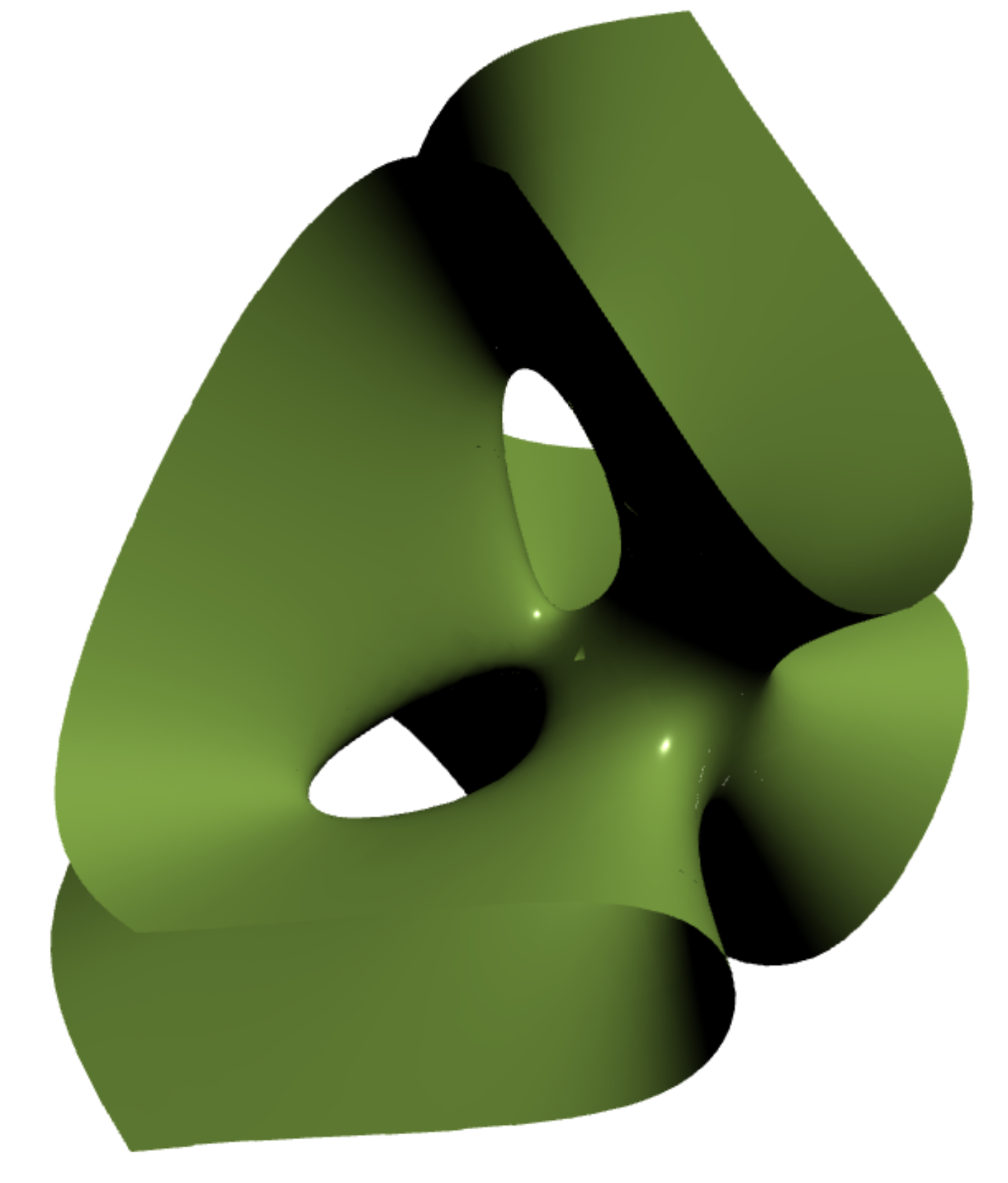}
        \caption{The Clebsch surface. }
        \label{fig:curvesb}
    \end{subfigure}
    \caption{Algebraic curves in the plane $(n=2)$ and an algebraic surface $(n=3)$.}
\end{figure}

\vspace{-0.4cm}
\end{example}
\noindent The set of solutions of the polynomial system \eqref{eq:polsys} is called an \emph{affine variety}. We denote this by $V_{\Kbar}(f_1, \ldots, f_s) = \{ x \in \Kbar^n ~|~ f_1(x) = \cdots = f_s(x) = 0 \}$, and replace $\Kbar$ by $K$ in this notation to mean only the solutions over the ground field. Examples of affine varieties are the red curve $V_{\Kbar}(f)$ and the set of black dots $V_{\Kbar}(f,g)$ in Fig.~\ref{fig:curvesa}. In the case of $V_{\Kbar}(f)$, Fig.~\ref{fig:curvesa} only shows the real part $V_{\Kbar}(f) \cap \mathbb{R}^2$.

\begin{example}{Example: surfaces in $\mathbb{R}^3$}
Let $K = \mathbb{R}$ and consider the affine variety $V = V_{\mathbb{C}}(f)$ where 
\begin{align} \label{eq:fclebsch}
\begin{split}
    f \, =&  \, \, 81 (x^3+y^3+z^3)-189 (x^2 y+x^2 z+y^2 x+y^2 z+x z^2+y z^2) +54  x y z \\ &  +126 (x y+x z+y z)-9 (x^2+y^2+z^2)-9 (x+y+z)+1. 
\end{split}
\end{align}
Its real part $V_{\mathbb{R}}(f)$ is the surface shown in Fig.~\ref{fig:curvesb}. The variety $V$ is called the \emph{Clebsch surface}. It is a \emph{cubic} surface because it is defined by an equation of degree three. We will revisit this surface in Section \ref{sec:lines}. Note that $f$ is invariant under permutations of the variables, i.e., $f(x,y,z)=f(y,x,z)=f(z,y,x)=f(x,z,y)=f(z,x,y)=f(y,z,x)$. This reflects in the symmetries of the surface $V_{\mathbb{R}}(f)$. Many polynomials from applications have similar symmetry properties. Exploiting this in computations is an active area of research, see for instance \cite{hubert:hal-03209117}.
\end{example}
More pictures of real affine varieties can be found, for instance, in \cite[Chapter 1, \S2]{cox2013ideals}, or in the algebraic surfaces gallery hosted at 
\begin{center}
\url{https://homepage.univie.ac.at/herwig.hauser/bildergalerie/gallery.html}.
\end{center}

We now briefly discuss commonly used fields $K$. In many engineering applications, the coefficients of $f_1, \ldots, f_s$ lie in $\mathbb{R}$ or $\mathbb{C}$. Computations in such fields use floating point arithmetic, yielding approximate results. The required quality of the approximation depends on the application. Other fields also show up: polynomial systems in cryptography often use $K = \mathbb{F}_q$, see for instance \cite{sala2009grobner}. Equations of many prominent algebraic varieties have integer coefficients, i.e., $K=\mathbb{Q}$. Examples are determinantal varieties (e.g., the variety of all $m \times n$ matrices of rank $< \min(m,n)$), Grassmannians in their Pl\"ucker embedding \cite[Chapter 5]{michalek2021invitation}, discriminants and resultants \cite[Sections 3.4, 5.2]{telen2020thesis} and toric varieties obtained from monomial maps \cite[Section 2.3]{telen2022introduction}. In number theory, one is interested in studying \emph{rational points} $V_{\mathbb{Q}}(f_1, \ldots, f_s) \subset V_{\overline{\mathbb{Q}}}(f_1, \ldots, f_s)$ on varieties defined over $\mathbb{Q}$. Recent work in this direction for del Pezzo surfaces can be found in \cite{mitankin2020rational,DESJARDINS2022108489}. Finally, in \emph{tropical geometry}, coefficients come from \emph{valued fields} such as the $p$-adic numbers $\mathbb{Q}_p$ or the Puiseux series $\mathbb{C} \{ \! \{ t \} \! \}$ \cite{maclagan2021introduction}. Solving over the field of Puiseux series is also relevant for \emph{homotopy continuation methods}, see Section \ref{subsec:homotopy}. We end the section with two examples highlighting the difference between $V_K(f_1, \ldots, f_s)$ and $V_{\overline{K}}(f_1, \ldots, f_s)$.

\begin{example}{Example: Fermat's last theorem}
Let $k \in \mathbb{N} \setminus \{0\}$ be a positive integer and consider the equation $f = x^k+y^k-1 = 0$. For any $k$, the variety $V_{\overline{\mathbb{Q}}}(f)$ has infinitely many solutions in $\overline{\mathbb{Q}}^2$. For $k = 1, 2$, there are infinitely many \emph{rational} solutions, i.e.~solutions in $\mathbb{Q}^2$. For $k \geq 3$, the only solutions in $\mathbb{Q}^2$ are $(1,0),(0,1)$ and, when $k$ is even, $(-1,0),(0,-1)$ \cite{darmon1995fermat}.
\end{example}

\begin{example}{Example: computing real solutions}
The variety $V_{\mathbb{C}}(x^2+y^2)$ consists of the two lines $x+\sqrt{-1} \cdot  y = 0$ and $x - \sqrt{-1} \cdot y = 0$ in $\mathbb{C}^2$. However, the real part $V_{\mathbb{R}}(x^2+y^2) = \{(0,0)\}$ has only one point. If we are interested only in this real solution, we may replace $x^2+y^2$ with the two polynomials $x,y$, which have the property that $V_{\mathbb{R}}(x^2+y^2) = V_{\mathbb{R}}(x,y) = V_{\mathbb{C}}(x,y)$. After this replacing step, an algorithm that computes all \emph{complex} solutions will still recover only the interesting solutions. It turns out that such a `better' set of equations can always be computed. The new polynomials generate the \emph{real radical ideal} associated to the original equations \cite[Sec.~6.3]{michalek2021invitation}. For recent computational progress, see \cite{baldi2021computing}. A different approach for real root finding in bounded domains is \emph{subdivision}, see \cite{mourrain2009subdivision}.
\end{example}

\section{Number of solutions} \label{sec:nosolutions}
A univariate polynomial $f \in \mathbb{C}[x]$ of degree $d$ has at most $d$ roots in $\mathbb{C}$. Moreover, $d$ is the \emph{expected} number of roots. We now formalize this. Consider the \emph{family}
\begin{equation} \label{eq:Fd}
    {\cal F}(d) \, = \, \{ a_0 + a_1 x + \cdots + a_d x^d ~|~ (a_0,\ldots, a_d) \in \mathbb{C}^{d+1} \} \, \simeq \, \mathbb{C}^{d+1} 
\end{equation}
of polynomials of degree at most $d$. There is an affine variety $\nabla_d \subset \mathbb{C}^{d+1}$, such that all $f \in {\cal F}(d) \setminus \nabla_d$ have precisely $d$ roots in $\mathbb{C}$. Here $\nabla_d = V_{\mathbb{C}}(\Delta_d)$, where $\Delta_d$ is a polynomial in the coefficients $a_i$ of $f \in {\cal F}(d)$. Equations for small $d$ are
\begin{align*}
\Delta_1 &= a_1,\\
   \Delta_2 &= a_2 \cdot ( a_{1}^{2}-4\,a_{0}a_{2} ),\\
   \Delta_3 &= a_3 \cdot ( a_{1}^{2}a_{2}^{2}-4\,a_{0}a_{2}^{3}-4\,a_{1}^{3}a_{3}+18\,a_{0}a_{1}a_{2}a_{3}-27\,a_{0}^{2}a_{3}^{2}), \\
   \Delta_4 &= a_4 \cdot( a_{1}^{2}a_{2}^{2}a_{3}^{2}-4\,a_{0}a_{2}^{3}a_{3}^{2}-4\,a_{1}^{3}a_{3}^{3}+18\,a_{0}a_{1}a_{2}a_{3}^{3}+ \cdots +256\,a_{0}^{3}a_{4}^{3}).
\end{align*}
Notice that $\Delta_d = a_d \cdot \widetilde{\Delta}_d$, where $\widetilde{\Delta}_d$ is the \emph{discriminant} for degree $d$ polynomials. Similar results exist for families of polynomial systems with $n > 1$, which bound the number of isolated solutions from above by the \emph{expected} number. This section states some of these results. \emph{It assumes that $K = \Kbar$ is algebraically closed.}

\subsection{B\'ezout's theorem} \label{subsec:bez}
Let $R = K[x] = K[x_1, \ldots, x_n]$. A \emph{monomial} in $R$ is a finite product of variables: $x^\alpha = x^{\alpha_1} \cdots x^{\alpha_n}$, $\alpha \in \mathbb{N}^n$. The \emph{degree} of the monomial $x^\alpha$ is $\deg(x^\alpha) = \sum_{i=1}^n \alpha_i$, and the degree of a polynomial $f = \sum_{\alpha} c_\alpha x^\alpha$ is $\deg(f) = \max_{ \{ \alpha \, : \,  c_\alpha \neq 0\} } \deg(x^\alpha)$. We define the vector subspaces 
\[ R_d \, = \, \{ f \in R \, : \, \deg(f) \leq d \}, \quad d  \in \mathbb{N}.\]
For an $n$-tuple of degrees $(d_1, \ldots, d_n)$, we define the family of polynomial systems
\[ {\cal F}(d_1, \ldots, d_n) \, = \, R_{d_1} \times \cdots \times R_{d_n}.\]
That is, $F = (f_1, \ldots, f_n) \in {\cal F}(d_1, \ldots, d_n)$ satisfies $\deg(f_i) \leq d_i, i = 1, \ldots, n$, and represents the polynomial system $F = 0$ with $s = n$. We leave the fact that ${\cal F}(d_1, \ldots, d_n) \simeq K^D$, with $D = \sum_{i=1}^n \binom{n+d_i}{n}$, as an exercise to the reader. Note that this is a natural generalization of \eqref{eq:Fd}. The set of solutions of $F = 0$ is denoted by $V_K(F) = V_K(f_1, \ldots, f_n)$, and a point in $V_K(F)$ is \emph{isolated} if it does not lie on a component of $V_K(F)$ with dimension $\geq 1$.

\newpage
\begin{theorem}[B\'ezout] \label{thm:bezout}
For any $F =(f_1, \ldots, f_n) \in {\cal F}(d_1, \ldots, d_n)$, the number of isolated solutions of $f_1 = \cdots = f_n = 0$, i.e., the number of isolated points in $V_K(F)$, is at most $d_1 \cdots d_n$. Moreover, there exists a proper subvariety $\nabla_{d_1,\ldots, d_n} \subsetneq K^D$ such that, when $F \in {\cal F}(d_1, \ldots, d_n) \setminus \nabla_{d_1, \ldots, d_n}$, $V_K(F)$ consists of exactly $d_1 \cdots d_n$ isolated~points.
\end{theorem}

The proof of this theorem can be found in \cite[Theorem III-71]{eisenbud2006geometry}. As in our univariate example, the variety $\nabla_{d_1, \ldots, d_n}$ can be described using discriminants and resultants. See, for instance, the discussion at the end of \cite[Section 3.4.1]{telen2020thesis}. Theorem \ref{thm:bezout} is an important result and gives an easy way to bound the number of isolated solutions of a system of $n$ equations in $n$ variables. The bound is \emph{almost always tight}, in the sense that the only systems with fewer solutions lie in $\nabla_{d_1,\ldots, d_n}$. Unfortunately, many systems coming from applications lie inside $\nabla_{d_1,\ldots, d_n}$. 
 Here is an example.

\begin{example}{Example: a planar robot arm}
This example comes from robotics. Consider a planar robot arm whose shoulder is fixed at the origin $(0,0)$ in the plane, and whose two arm segments have fixed lengths $L_1$ and $L_2$. We determine the possible positions of the elbow $(x,y)$, given that the hand of the robot touches a given point $(a,b)$. The situation is illustrated in Fig.~\ref{fig:robot}. The Pythagorean theorem gives the identities
\begin{equation} \label{eq:robotics} x^2 + y^2 - L_1^2 \, = \, (a-x)^2+(b-y)^2-L_2^2 \, = \, 0, \end{equation}
which is a system of $s = 2$ equations in $n = 2$ variables $x, y$. The plane curves corresponding to these equations are shown in Fig.~\ref{fig:robot}. Their intersection points are the possible configurations. Naturally, more complicated robots lead to more involved equations, see \cite{wampler2011numerical}. The system \eqref{eq:robotics} with $K = \mathbb{C}$ lies in $\nabla_{2,2}$: the two real solutions seen in Fig.~\ref{fig:robot} are the only solutions over $\mathbb{C}$, and $2 < d_1 \cdot d_2 = 4$. However, the slightest perturbation of the equations introduces two extra solutions. For instance, replace the first equation with $x^2 + \epsilon \cdot xy + y^2 -L_1^2 = 0$, for small $\epsilon$. The resulting system lies in ${\cal F}(2,2) \setminus \nabla_{2,2}$. It has four complex solutions, two of which lie close to the intersection points in Figure \ref{fig:robot}. The other two are \emph{large}, see Remark \ref{rem:projspace}.
\begin{figure}
    \centering
    \includegraphics[height = 4cm]{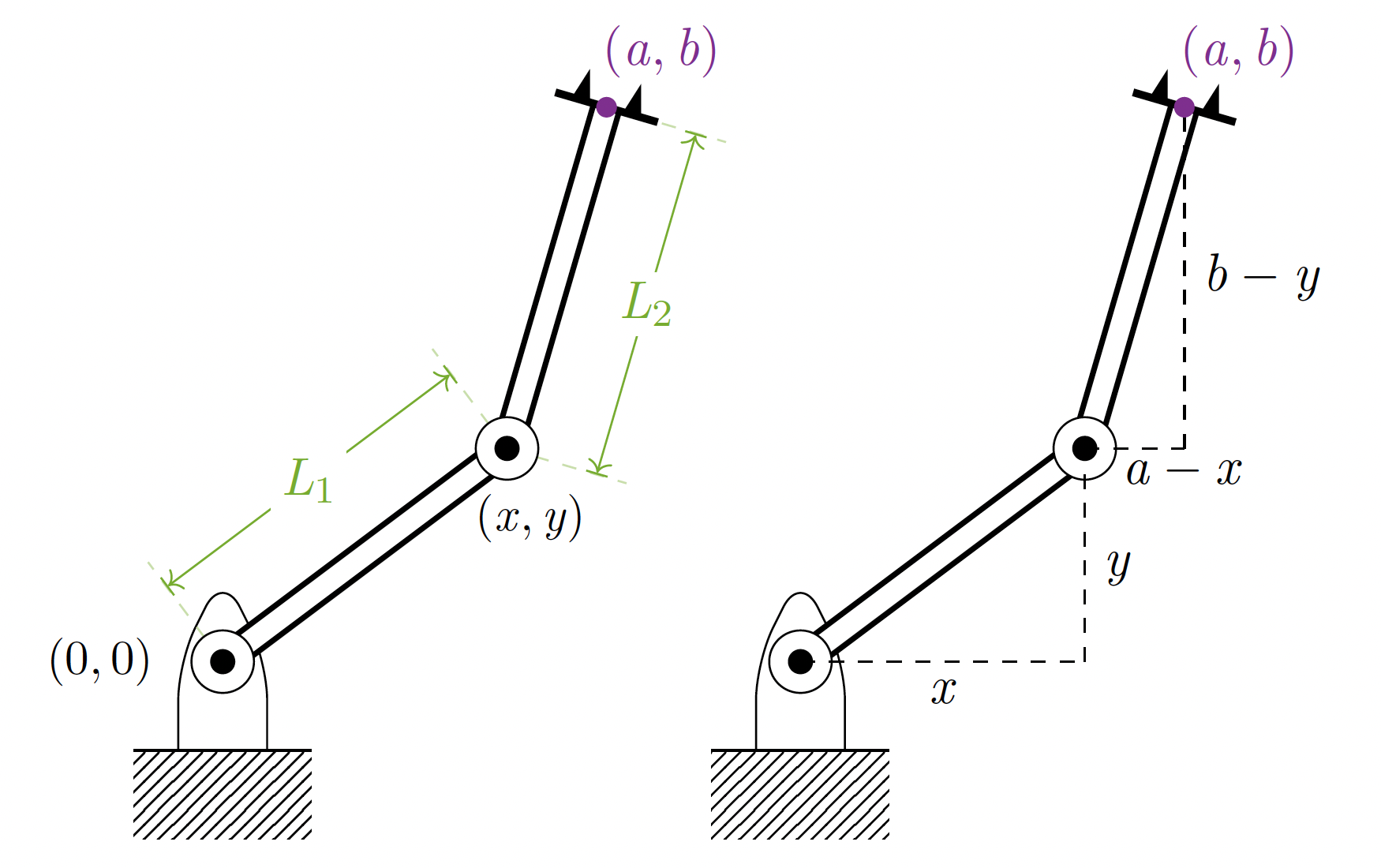}
    \quad
    \includegraphics[height = 5cm]{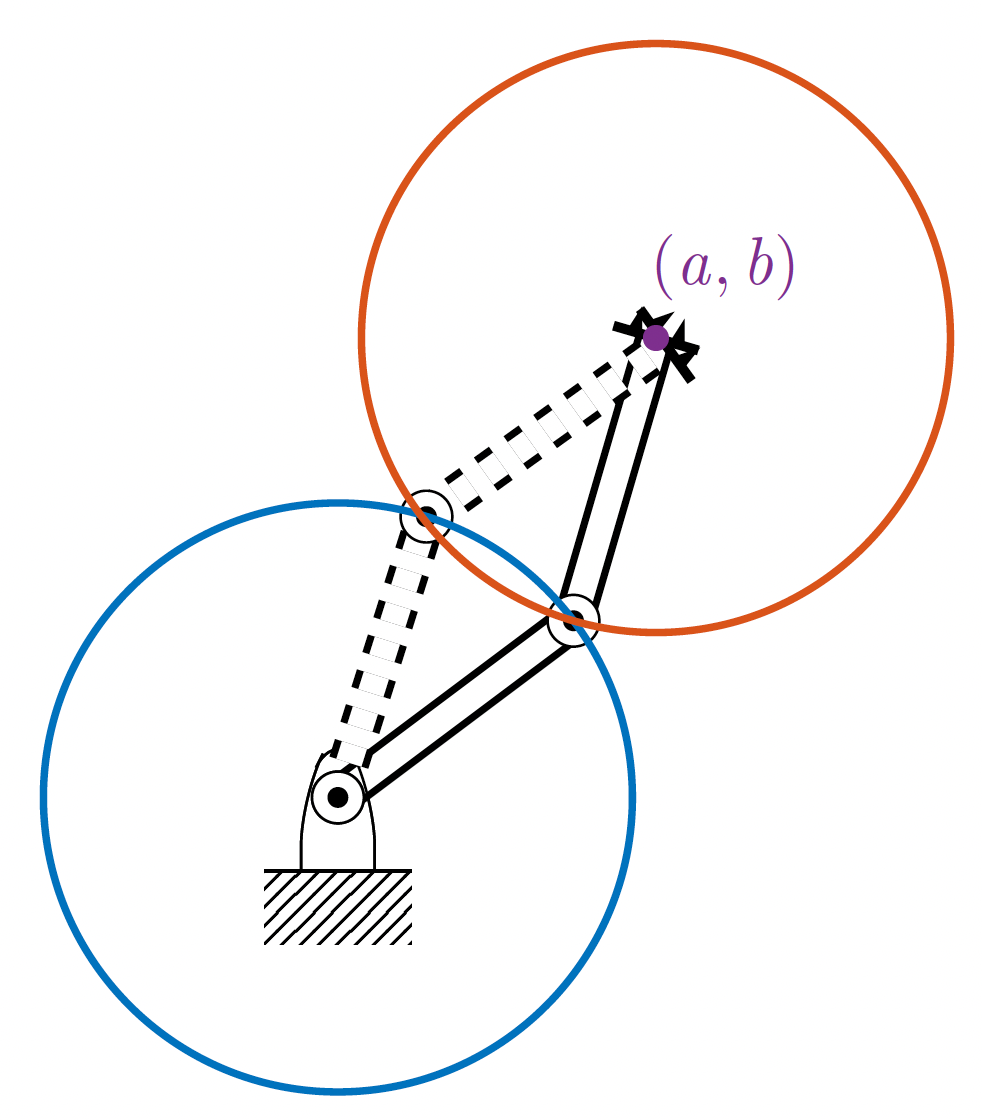}
    \caption{The two configurations of a robot arm are the intersection points of two circles.}
    \label{fig:robot}
\end{figure}
\end{example}

\begin{remark} \label{rem:projspace}
B\'ezout's theorem more naturally counts solutions in \emph{projective space} $\mathbb{P}_K^n$, and it accounts for solutions with \emph{multiplicity} $>1$. More precisely, if $f_i$ is a homogeneous polynomial in $n+1$ variables of degree $d_i$, and $f_1 = \cdots = f_n = 0$ has finitely many solutions in $\mathbb{P}_K^n$, the number of solutions (counted with multiplicity) is \emph{always} $d_1 \cdots d_n$. We encourage the reader who is familiar with projective geometry to check that \eqref{eq:robotics} defines two solutions \emph{at infinity}, when each of the equations is viewed as an equation on $\mathbb{P}^2_{\mathbb{C}}$ by homogenizing. Introducing $\epsilon$ brings these solutions back into $\mathbb{C}^2$. Since they come from infinity, they have large coordinates. 
\end{remark}

\subsection{Kushnirenko's theorem}
An intuitive consequence of Theorem \ref{thm:bezout} is that \emph{random} polynomial systems given by polynomials of fixed degree always have the same number of solutions. Looking at $f$ and $g$ from \eqref{eq:running}, we see that they do not look so \emph{random}, in the sense that some monomials of degree $\leq 3$ are missing. For instance, $x^3$ and $y^3$ do not appear. Having zero coefficients standing with some monomials in ${\cal F}(d_1, \ldots, d_n)$ is sometimes enough to conclude that the system lies in $\nabla_{d_1, \ldots, d_n}$. That is, the system is not \emph{random} in the sense of B\'ezout's theorem. 
The \emph{(monomial) support} of $f = \sum_{\alpha} c_\alpha x^\alpha$~is
\[ {\rm supp} \big (\sum_{\alpha} c_\alpha x^\alpha \big ) \, = \, \{ \alpha \, : \, c_\alpha \neq 0 \} \, \subset \, \mathbb{N}^n.\]
This subsection considers families of polynomial systems with fixed support. Let ${\cal A} \subset \mathbb{N}^n$ be a finite subset of exponents of cardinality $|{\cal A}|$. We define 
\[ {\cal F}({\cal A}) \, = \, \{ \, (f_1, \ldots, f_n) \in R^n \, : \, {\rm supp}(f_i) \subset {\cal A}, \, i = 1, \ldots, n \, \} \, \simeq \, K^{\, n \cdot | {\cal A} |}. \]
The next theorem expresses the number of solutions for systems in this family in terms of the volume ${\rm Vol}({\cal A}) = \int_{{\rm Conv}({\cal A})} {\rm d} \alpha_1 \cdots {\rm d} \alpha_n$ of the convex polytope
\begin{equation} \label{eq:conv} {\rm Conv}({\cal A}) \, = \, \left \{ \sum_{\alpha \in {\cal A}} \lambda_\alpha \cdot \alpha \, : \, \lambda_\alpha \geq 0, \, \sum_{\alpha \in {\cal A}} \lambda_\alpha = 1  \right \} \, \,  \subset \mathbb{R}^n. \end{equation}
The \emph{normalized volume} ${\rm vol}({\cal A})$ is defined as $n! \cdot {\rm Vol}({\cal A})$.
\begin{theorem}[Kushnirenko] \label{thm:kushnirenko}
For any $F = (f_1, \ldots, f_n) \in {\cal F}({\cal A})$, the number of isolated solutions of $f_1 = \cdots = f_n = 0$ in $(K \setminus \{0\})^n$, i.e., the number of isolated points in $V_K(F) \cap (K \setminus \{0\})^n$, is at most ${\rm vol}({\cal A})$. Moreover, there exists a proper subvariety $\nabla_{\cal A} \subsetneq K^{\, n \cdot | {\cal A} |}$ such that, when $F \in {\cal F}({\cal A}) \setminus \nabla_{\cal A}$, $V_K(F) \cap (K \setminus \{0\})^n$ consists of ${\rm vol}({\cal A})$ isolated points.
\end{theorem}
For a proof, see \cite{kushnirenko1976newton}. The theorem necessarily counts solutions in $(K \setminus \{0\})^n \subset K^n$, as multiplying all equations with a monomial $x^\alpha$ may change the number of solutions in the coordinate hyperplanes (i.e., there may be new solutions with zero-coordinates). However, it does not change the normalized volume ${\rm vol}({\cal A})$. The statement can be adapted to count solutions in $K^n$, but becomes more involved \cite{huber1997bernstein}. We point out that, with the extra assumption that $0 \in {\cal A}$, one may replace $(K \setminus \{0\})^n$ by $K^n$ in Theorem \ref{thm:kushnirenko}.
To compare Kushnirenko's theorem with B\'ezout, note that ${\cal F}({\cal A})$ for
\begin{equation} \label{eq:Abez}
{\cal A} = \{ \alpha \in \mathbb{N}^n  :  \deg(x^\alpha) \leq d \} \end{equation}
is $ {\cal F}(d, \ldots, d)$, and $d^{\, n} = {\rm vol}({\cal A})$. Theorem \ref{thm:kushnirenko} recovers Theorem \ref{thm:bezout} for $d_1= \cdots = d_n$.

\begin{example}{Example: back to plane curves} 
The polynomial system $f = g = 0$ from \eqref{eq:running} belongs to the family ${\cal F}({\cal A})$ with ${\cal A} = \{ (1,0), (0,1), (2,0), (1,1), (0,2), (2,1), (1,2) \}$. The convex hull ${\rm Conv}({\cal A})$ is a hexagon in $\mathbb{R}^2$, see Fig.~\ref{fig:hexagon}. Its normalized volume is ${\rm vol}({\cal A}) = n! \cdot {\rm Vol}({\cal A}) = 2! \cdot 3 = 6$. Theorem \ref{thm:kushnirenko} predicts six solutions in $(\overline{\mathbb{Q}} \setminus \{0\})^2$. These are six of the seven black dots seen in the left part of Fig.~\ref{fig:curvesa}: the solution $(0,0)$ is not counted. We have a chain of inclusions $\nabla_{\cal A} \subset {\cal F}({\cal A}) \subset \nabla_{3,3} \subset {\cal F}(3,3)$ and $(f,g) \in {\cal F}({\cal A}) \setminus \nabla_{\cal A}$.
\begin{figure}
    \centering
    \input{hexagon}
    \caption{The polytope ${\rm Conv}({\cal A})$ in this example is a hexagon.}
    \label{fig:hexagon}
\end{figure}
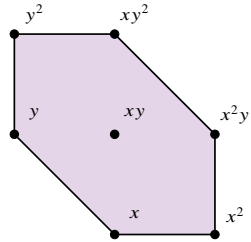
\end{example}

\begin{remark}
The analog of Remark \ref{rem:projspace} for Theorem \ref{thm:kushnirenko} is that ${\rm vol}({\cal A})$ counts solutions on the \emph{projective toric variety} $X_{\cal A}$ associated with ${\cal A}$. It equals the degree of $X_{\cal A}$ in its embedding in $\mathbb{P}_K^{|{\cal A}|-1}$ (after multiplying with a lattice index). A proof and examples are given in \cite[Section~3.4]{telen2022introduction}. When ${\cal A}$ is as in \eqref{eq:Abez}, we have $X_{\cal A}= \mathbb{P}^n$. 
\end{remark}

\begin{remark}
The convex polytope ${\rm Conv}({\rm supp}(f))$ is called the \emph{Newton polytope} of $f$. Its importance goes beyond counting solutions: it is dual to the \emph{tropical hypersurface} defined by $f$, which is a \emph{combinatorial shadow} of $V_K(f) \cap (K \setminus \{0\})^n$ \cite[Prop.~3.1.6]{maclagan2021introduction}.
\end{remark}

 \subsection{Bernstein's theorem} \label{subsec:bkk}
There is a generalization of Kushnirenko's theorem which allows different supports for the polynomials $f_1, \ldots, f_n$. We fix $n$ finite subsets of exponents ${\cal A}_1, \ldots, {\cal A}_n \subset \mathbb{N}^n$ with respective cardinalities $|{\cal A}_i|$. These define the family of polynomial systems 
\[ {\cal F}({\cal A}_1, \ldots, {\cal A}_n) \, = \, \{ \, (f_1, \ldots, f_n) \in R^n \, : \, {\rm supp}(f_i) \subset {\cal A}_i, \, i = 1, \ldots, n \, \} \, \simeq \, K^D, \]
where $D = |{\cal A}_1| + \cdots + |{\cal A}_n|$.
The number of solutions is characterized by the \emph{mixed volume} of ${\cal A}_1, \ldots, {\cal A}_n$, which we now define. The \emph{Minkowski sum} $S + T$ of two sets $S, T \subset \mathbb{R}^n$ is $\{ s + t \, : \, s \in S, t \in T \}$, where $s + t$ is the usual addition of vectors in $\mathbb{R}^n$. For a nonnegative real number $\lambda$, the \emph{$\lambda$-dilation} of $S \subset \mathbb{R}^n$ is $\lambda \cdot S = \{ \lambda \cdot s \, : \, s \in S \}$, where $\lambda \cdot s$ is the usual scalar multiplication in $\mathbb{R}^n$.   Each of the supports ${\cal A}_i$ gives a convex polytope ${\rm Conv}({\cal A}_i)$ as in \eqref{eq:conv}. The function $\mathbb{R}_{\geq 0}^n \rightarrow \mathbb{R}_{\geq 0}$ given by 
\begin{equation} \label{eq:volpol}
(\lambda_1, \ldots, \lambda_n) \longmapsto {\rm Vol}( \, \lambda_1 \cdot {\rm Conv}({\cal A}_1) \, + \, \cdots \, + \, \lambda_n \cdot {\rm Conv}({\cal A}_n) \, ) 
\end{equation}
is a homogeneous polynomial of degree $n$, meaning that all its monomials have degree $n$ \cite[Chapter 7, \S 4, Proposition 4.9]{cox2006using}. The \emph{mixed volume} ${\rm MV}({\cal A}_1, \ldots, {\cal A}_n)$ is the coefficient of the polynomial \eqref{eq:volpol} standing with the monomial $\lambda_1 \cdots \lambda_n$. 

\begin{theorem}[Bernstein-Kushnirenko] \label{thm:BKK}
For any $F = (f_1, \ldots, f_n) \in {\cal F}({\cal A}_1, \ldots, {\cal A}_n)$, the number of isolated solutions of $f_1 = \cdots = f_n = 0$ in $(K \setminus \{0\})^n$, i.e., the number of isolated points in $V_K(F) \cap (K \setminus \{0\})^n$, is at most ${\rm MV}({\cal A}_1, \ldots, {\cal A}_n)$. Moreover, there exists a proper subvariety $\nabla_{{\cal A}_1, \ldots, {\cal A}_n} \subset K^D$ such that, when $F \in {\cal F}({\cal A}_1, \ldots, {\cal A}_n) \setminus \nabla_{{\cal A}_1, \ldots, {\cal A}_n}$, $V_K(F) \cap (K \setminus \{0\})^n$ consists of precisely ${\rm MV}({\cal A}_1, \ldots, {\cal A}_n)$ isolated points. 
\end{theorem}
This theorem was originally proved by Bernstein for $K = \mathbb{C}$ in \cite{bernshtein1975number}. The proof by Kushnirenko in \cite{kushnirenko1976newton} works for algebraically closed fields. Several alternative proofs were found by Khovanskii \cite{khovanskii1978newton}. Theorem \ref{thm:BKK} is sometimes called the \emph{BKK theorem}, after the aforementioned mathematicians. Like Kushnirenko's theorem, Theorem \ref{thm:BKK} can be adapted to count roots in $K^n$ rather than $(K \setminus \{0\})^n$ \cite{huber1997bernstein}, and if $0 \in {\cal A}_i$ for all $i$, one may replace $(K \setminus \{0\})^n$ by $K^n$. 

When ${\cal A}_1 = \cdots = {\cal A}_n = {\cal A}$, we have ${\cal F}({\cal A}_1, \ldots, {\cal A}_n) = {\cal F}({\cal A})$, and when ${\cal A}_i = \{ \alpha \in \mathbb{N}^n : \deg(x^\alpha) \leq d_i \}$, we have ${\cal F}({\cal A}_1, \ldots, {\cal A}_n) = {\cal F}(d_1, \ldots, d_n)$. Hence, all families of polynomials we have seen before are of this form, and Theorem \ref{thm:BKK} generalizes Theorems \ref{thm:bezout} and \ref{thm:kushnirenko}. Note that, in particular, we have ${{\rm MV}({\cal A}, \ldots, {\cal A}) = {\rm vol}({\cal A})}$.

\begin{example}{Example: mixed areas}
A useful formula for $n = 2$ is ${\rm MV}({\cal A}_1, {\cal A}_2) = {\rm Vol}({\cal A}_1 + {\cal A}_2) - {\rm Vol}({\cal A}_1) - {\rm Vol}({\cal A}_2)$. For instance, the following two polynomials appear in \cite[Example 5.3.1]{telen2020thesis}:
\begin{align*}
f = a_0 + a_1x^3y + a_2 xy^3, \quad g = b_0 + b_1x^2 + b_2y^2 + b_3x^2y^2.
\end{align*}
The system $f = g = 0$ is a general member of the family ${\cal F}({\cal A}_1, {\cal A}_2) \simeq K^7$, where ${\cal A}_1 = \{(0,0),(3,1),(1,3) \}$ and ${\cal A}_2 = \{ (0,0),(2,0),(0,2),(2,2) \}$. The Newton polygons, together with their Minkowski sum, are shown in Fig.~\ref{fig:sparseexample}. 
\begin{figure}
\centering
\includegraphics[height = 3cm]{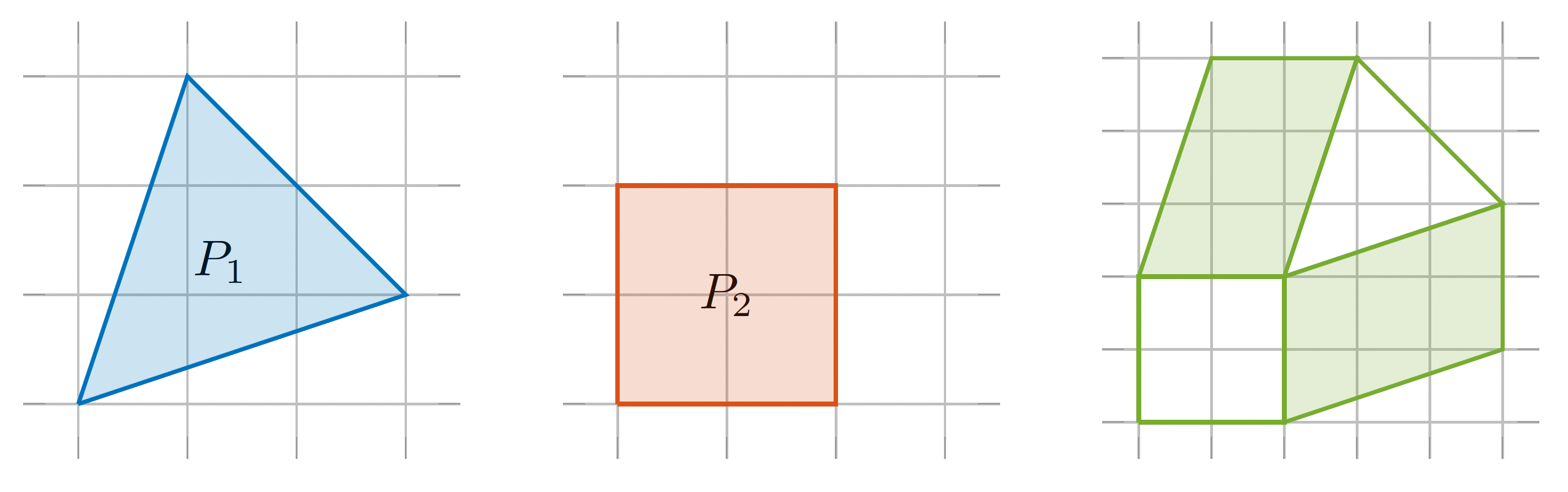}
\caption{The green area counts the solutions to equations with support in $P_1, P_2$.}
\label{fig:sparseexample}
\end{figure}
By applying~the formula for ${\rm MV}({\cal A}_1, {\cal A}_2)$ seen above, we find that the mixed volume for the system $f = g = 0$ is the green area in the right part of Fig.~\ref{fig:sparseexample}, which is 12. Note that the B\'ezout bound (Theorem \ref{thm:bezout}) is 16. and Theorem \ref{thm:kushnirenko} also predicts 12 solutions, with ${\cal A} = {\cal A}_1 \cup {\cal A}_2$. Hence ${\cal F}({\cal A}_1, {\cal A}_2) \subset {\cal F}({\cal A}) \subset \nabla_{4,4} \subset {\cal F}(4,4)$ and ${{\cal F}({\cal A}_1, {\cal A}_2) \not \subset \nabla_{\cal A}}$. 
\end{example}
Theorem \ref{thm:BKK} provides an upper bound on the number of isolated solutions to any system of polynomial equations with $n = s$. Although it improves significantly on B\'ezout's bound for many systems, it still often happens that the bound is not tight for systems in applications. That is, one often encounters systems $F \in \nabla_{{\cal A}_1, \ldots, {\cal A}_n}$. Even more refined root counts exist, such as those based on \emph{Newton-Okounkov bodies} \cite{kaveh2012newton}. In practice, with today's computational methods (see Section \ref{sec:methods}), we often count solutions reliably by simply solving the system. \emph{Certification methods} provide a proof for a \emph{lower} bound on the number of solutions \cite{breiding2020certifying}. The actual number of solutions is implied if one can match this with a theoretical upper bound.

\section{Computational methods} \label{sec:methods}
We give a brief introduction to two of the most important computational methods for solving polynomial equations. The first method uses \emph{normal forms}, the second is based on \emph{homotopy continuation}. We keep writing $F = (f_1, \ldots, f_s) = 0$ for the system we want to solve. We require $s \geq n$, and assume finitely many solutions over $\Kbar$. All methods discussed here compute all solutions over $\Kbar$, \emph{so we keep assuming that $K = \Kbar$ is algebraically closed}. An important distinction between normal forms and homotopy continuation is that the former works over any field $K$, while the latter needs $K =\mathbb{C}$. If the coefficients are contained in a subfield (e.g.~$\mathbb{R} \subset \mathbb{C}$), a significant part of the computation in normal form algorithms can be done over this subfield. Also, homotopy continuation is most natural when $n = s$, whereas $s > n$ is not so much a problem for normal forms. However, if $K =\mathbb{C}$ and $n = s$, continuation methods are extremely efficient and can compute millions of solutions.

\subsection{Normal form methods}
Let $I = \langle f_1, \ldots, f_s \rangle \subset R = K[x_1, \ldots, x_n]$ be the ideal generated by our polynomials. For ease of exposition, we assume that $I$ is \emph{radical}, which is equivalent to all points in $V_{K}(I) = V_{K}(f_1, \ldots, f_s)$ having multiplicity one. In other words, the Jacobian matrix $(\partial f_i/\partial x_j)$, evaluated at any of the points in $V_{K}(I)$, has rank $n$. Let us write $V_{K}(I) = \{ z_1, \ldots, z_\delta \} \subset K^n$ for the set of solutions, and $R/I$ for the quotient ring obtained from $R$ by the equivalence relation $f \sim g \Leftrightarrow f-g \in I$. The main observation behind normal form methods is that the coordinates of $z_i$ are encoded in the eigenstructure of the $K$-linear endomorphisms $M_g: R/I \rightarrow R/I$ given by $[f] \mapsto [g \cdot f]$, where $[f]$ is the residue class of $f$ in $R/I$. 

We will now make this precise. First, we show that $\dim_K R/I = \delta$.
We define ${\rm ev}_i : R/I \rightarrow K$ as ${\rm ev}_i ([f]) = f(z_i)$, and combine these to get
\[ {\rm ev} = (\, {\rm ev}_1, \ldots, {\rm ev}_\delta \, ) \, : \, R/I \longrightarrow K^{ \, \delta}, \quad \text{given by } \quad {\rm ev}([f]) = ( \,  f(z_1), \ldots, f(z_\delta) \, ).\]
By Hilbert's Nullstellensatz  \cite[Chapter 4]{cox2013ideals}, a polynomial $f \in R$ belongs to $I$ if and only if $f(z_i) = 0, i = 1, \ldots, \delta$. In other words, the map ${\rm ev}$ is injective. It is also surjective: there exist \emph{Lagrange polynomials} $\ell_i \in R$ satisfying $\ell_i(z_j) = 1$ if $i = j$ and $\ell_i(z_j) = 0$ for $i \neq j$ \cite[Lemma 3.1.2]{telen2020thesis}. We conclude that $R/I \overset{\rm ev}{\simeq} K^\delta$. 

The following statement makes our claim that \emph{the zeros $z_1, \ldots, z_\delta$ are encoded in the eigenstructure of $M_g$} concrete. 
\begin{theorem} \label{thm:EVthm}
The left eigenvectors of the $K$-linear map $M_g$ are the evaluation functionals ${\rm ev}_i, i = 1, \ldots, \delta$. The eigenvalue corresponding to ${\rm ev}_i$ is $g(z_i)$. 
\end{theorem}
\begin{proof}
We have $({\rm ev}_i \circ M_g )([f]) = {\rm ev}_i([g \cdot f]) = g(z_i) f(z_i) = g(z_i) \cdot {\rm ev}_i([f])$, which shows that ${\rm ev}_i$ is a left eigenvector with eigenvalue $g(z_i)$. Moreover, the ${\rm ev}_i$ form a complete set of eigenvectors, since ${\rm ev}: R/I \rightarrow K^\delta$ is a $K$-linear isomorphism.
\end{proof}
We encourage the reader to check that the residue classes of the Lagrange polynomials $[\ell_i] \in R/I$ form a complete set of right eigenvectors. We point out that, after~choosing a basis of $R/I$, the functional ${\rm ev}_i$ is represented by a row vector $w_i^\top$ of length $\delta$, and $M_g$ is a \emph{multiplication matrix} of size $\delta \times \delta$. The eigenvalue relation in the proof of Theorem \ref{thm:EVthm} reads more familiarly as $w_i^\top M_g = g(z_i) \cdot w_i^\top$. 
Theorem \ref{thm:EVthm} suggests breaking up the task of computing $V_K(I) = \{ z_i \}_{i=1}^\delta$ into two~parts: 
\begin{enumerate}[leftmargin=.5in]
\item[(A)] Compute multiplication matrices $M_g$ and
\item[(B)] extract the coordinates of $z_i$ from their eigenvectors or eigenvalues. 
\end{enumerate}
For step (B), let $\{ [b_1] , \ldots, [b_\delta] \}$ be a $K$-basis for $R/I$, with $b_j \in R$. The vector $w_i$ is explicitly given by $w_i = (b_1(z_i), \ldots, b_\delta(z_i))$. If the coordinate functions $x_1, \ldots, x_n$ are among the $b_j$, one reads the coordinates of $z_i$ directly from the entries of $w_i$. If not, some more processing might be needed. Alternatively, one can choose $g = x_j$ and read the $j$-th coordinates of the $z_i$ from the eigenvalues of $M_{x_j}$. There are many things to say about these procedures, in particular about their efficiency and numerical stability. We refer the reader to \cite[Remark 4.3.4]{telen2020thesis} for references and more details, and do not elaborate on this here. 

We turn to step (A), which is where \emph{normal forms} come into play. Suppose a basis $\{ [b_1] , \ldots, [b_\delta] \}$ of $R/I$ is fixed. We identify $R/I$ with $B = {\rm span}_K(b_1, \ldots, b_\delta) \subset R$. For any $f \in R$, there are unique constants $c_j(f) \in K$ such that 
\begin{equation} \label{eq:NF}
    f - \sum_{j=1}^\delta c_j(f) \cdot b_j \in I.
\end{equation}  
These are the coefficients in the unique expansion of $[f] = \sum_{j=1}^\delta c_j(f) \cdot [b_j]$ in our basis. The $K$-linear map ${\cal N}: R \rightarrow B$ which sends $f$ to $\sum_{j=1}^\delta c_j(f) \cdot b_j$ is called a \emph{normal form}. Its key property is that ${\cal N}$ \emph{projects} $R$ onto $B$ \emph{along} $I$, meaning that ${\cal N} \circ {\cal N} = {\cal N}$ (${\cal N}_{|B}$ is the identity), and $\ker {\cal N} = I$. The multiplication map $M_g: B \rightarrow B$ is simply given by $M_g(b) = {\cal N}(g \cdot b)$. More concretely, the $i$-th column of the matrix representation of $M_g$ contains the coefficients $c_j(g \cdot b_i), j = 1, \ldots, \delta$ of ${\cal N}(g \cdot b_i)$. Here is a familiar example. 

\begin{example}{Example: normal forms for $n = 1$}
Let $I = \langle f \rangle = \langle a_0 + a_1 x + \cdots + a_d x^d \rangle$ be the ideal generated by the univariate polynomial $f \in K[x]$. For general $a_i$, there are $\delta = d$ roots with multiplicity one, hence $I$ is radical. The dimension $\dim_K K[x]/I$ equals $d$, and a canonical choice of basis is $\{ [1], [x], \ldots, [x^{d-1}]\}$. Let us construct the matrix $M_x$ in this basis. That is, we set $g = x$. We compute the normal forms ${\cal N}(x \cdot x^{i-1})$:
\[{\cal N}(x^i) = x^i, \,  i = 1, \ldots, d-1 \quad \text{and} \quad {\cal N}(x^d) = -a_d^{-1}(a_0 + a_1x + \cdots + a_{d-1} x^{d-1}). \]
One checks this by verifying that $x^i - {\cal N}(x^i) \in \langle f \rangle$. The coefficients $c_j(x^i), j = 1, \ldots, d$ of ${\cal N}(x^i)$ form the $i$-th column of the companion matrix $C_f$ in \eqref{eq:companionmatrix}. Hence $M_x = C_f$, and Theorem \ref{thm:EVthm} confirms that the eigenvalues of $C_f$ are the roots of $f$. 
\end{example}

Computing normal forms can be done using linear algebra on certain structured matrices, called \emph{Macaulay matrices}. We illustrate this with an example from \cite{telen2018stabilized}. 

\begin{example}{Example: Macaulay matrices} 
Consider the ideal $I = \langle f, g \rangle \subset \mathbb{Q}[x,y]$ given by $f= x^2+y^2 -2,g = 3x^2 -y^2 -2$. The variety $V_{\overline{\mathbb{Q}}}(I) = V_{\mathbb{Q}}(I)$ consists of 4 points $\{(-1,-1),(-1,1),(1,-1),(1,1)\}$, as predicted by Theorem \ref{thm:bezout}. 
We construct a \emph{Macaulay matrix} whose rows are indexed by $f,xf,yf,g,xg,yg$, and whose columns are indexed by all monomials of degree~$\leq 3$: 
\[
  \cal M=\kbordermatrix{%
      &  x^3 & x^2y & xy^2 & y^3 &x^2 & y^2 && 1  & x & y &  xy   \\
    f &   &  &  &  & 1  & 1 &\vrule& -2   \\
    xf & 1 &  & 1 &      &  &&\vrule& & -2 &  &  \\
    yf &  & 1  &  & 1    &   &&\vrule&   &  &-2&  \\
    g &   &  &  &&  3  & -1 &\vrule&-2 \\
    xg & 3  &  & -1 &   & &   &\vrule& & -2 &  &  \\
    yg &  & 3  &  & -1   &   & &\vrule&   &  & -2 & 
  }.
\]
The first row reads $f = 1 \cdot x^2 + 1 \cdot y^2 -2 \cdot 1$. A basis for $\mathbb{Q}[x,y]/I$ is $\{[1], [x],[y],[xy]\}$. These monomials index the last four columns. 
We now invert the leftmost $6 \times 6$ block and apply this inverse from the left to ${\cal M}$:
\[
  \tilde{\cal M}=\kbordermatrix{%
      &  x^3 & x^2y & xy^2 & y^3 &x^2 & y^2 && 1  & x & y &  xy   \\
     x^3-x & 1 &  &  &  &  &  &\vrule& & -1  &  &\\
     x^2y-y&   & 1&  &  &  &  &\vrule&&  & -1 &  \\
     xy^2-x &   &  & 1&  &  &  &\vrule&  &-1&  &  \\
     y^3-y &   &  &  & 1&  &  &\vrule&  &  &-1&\\
     x^2-1&   &  &  &  &1 &  &\vrule& -1 &&  &  \\
     y^2-1&   &  &  &  &  &1 &\vrule& -1 &  && 
  }.
\]
The rows of $\tilde{\cal M}$ are linear combinations of the rows of $\cal M$, representing polynomials in $I$. The first row reads $x^3 - 1 \cdot x \in I$. Comparing this with \eqref{eq:NF}, we see that we have found that ${\cal N}(x^3) = x$. Using $\tilde{\cal M}$ we can construct $M_x$ and $M_y$:
\[
  M_x =\kbordermatrix{%
   &{[x]}&{[x^2]}&{[xy]}&{[x^2y]} \\
    {[1]} & 0  & 1  & 0 & 0 \\
    {[x]} & 1 & 0 &  0  & 0\\
    {[y]} & 0  & 0  & 0  & 1\\
    {[xy]} & 0  & 0 & 1 & 0\\
  }, \quad 
  M_y =\kbordermatrix{%
   &{[y]}&{[xy]}&{[y^2]}&{[xy^2]} \\
    {[1]} & 0  & 0  & 1 & 0 \\
    {[x]} & 0 & 0 &  0  & 1\\
    {[y]} & 1  & 0  & 0  & 0\\
    {[xy]} & 0  & 1 & 0 & 0\\
  }. 
\]
The reader is encouraged to verify Theorem \ref{thm:EVthm} for these matrices. 
\end{example}
\begin{remark}
The entries of a Macaulay matrix $\cal M$ are the coefficients of the polynomials $f_1, \ldots, f_s$. An immediate consequence of the fact that normal forms are computed using linear algebra on Macaulay matrices is that when the coefficients of $f_i$ are contained in a subfield $\tilde{K} \subset K$, all computations in step $(A)$ can be done over $\tilde{K}$.  This assumes the polynomials $g$ for which we want to compute $M_g$ have coefficients~in~$\tilde{K}$. 
\end{remark}

As illustrated in the example above, to compute the matrices $M_g$ it is sufficient to determine the restriction of the normal form ${\cal N}: R \rightarrow B$ to a finite-dimensional $K$-vector space $V \subset R$, containing $g \cdot B$. The restriction ${\cal N}_{|V} : V \rightarrow B$ is called a \emph{truncated normal form}, see \cite{telen2018solving} and \cite[Chapter 4]{telen2020thesis}. The dimension of the space $V$ counts the number of columns of the Macaulay matrix.

Usually, one chooses the basis elements $b_j$ of $B$ to be monomials, and $g$ to be a coordinate function $x_i$. The basis elements may arise as standard monomials from a \emph{Gr\"obner basis} computation. We briefly discuss this important concept.

\begin{important}{Gr\"obner bases}
Gr\"obner bases are powerful tools for symbolic computation in algebraic geometry. A nice way to motivate their definition is by considering \emph{Euclidean division} as a candidate for a normal form map. In the case $n=1$, this rewrites $g \in K[x]$ as 
\begin{equation} \label{eq:univeuclid}
    g = q \cdot f + r, \quad \text{where} \quad \deg(r)<d.
\end{equation} 
Clearly $[g] = [r]$ in $K[x]/\langle f \rangle$, and ${\cal N}(g) = r$ since $r \in B = {\rm span}_K(1,x, \ldots, x^{d-1})$.

To generalize this to $n>1$ variables, we fix a \emph{monomial order} $\preceq$ on $R = K[x_1, \ldots, x_n]$ and write ${\rm LT}(f)$ for the \emph{leading term} of $f$ with respect to $\preceq$. The reader who is unfamiliar with monomial orders can consult \cite[Chapter 2, \S 2]{cox2013ideals}. As above, let $I = \langle f_1, \ldots, f_s \rangle \subset R$ be a radical ideal such that $|V_K(I)|=\delta <\infty$. As basis elements $b_1, \ldots, b_\delta$ of $B \simeq R/I$, we use the $\delta$ $\preceq$--smallest monomials which are linearly independent modulo our ideal $I$. They are also called \emph{standard monomials}. By \cite[Chapter 9, \S 3, Theorem 3]{cox2013ideals}, there exists an algorithm which, for any input $g \in R$, computes $q_1, \ldots, q_s, r \in R$ such that
\begin{equation} \label{eq:multieuclid}
g = q_1 \cdot f_1 + \cdots + q_s \cdot f_s + r, \quad \text{where ${\rm LT}(f_i)$ does not divide any term of $r$, $\forall i$}.
\end{equation}
This algorithm is called \emph{multivariate Euclidean division}. Note how the condition ``${\rm LT}(f_i)$ does not divide any term of $r$, for all $i$'' generalizes $\deg(r)<d$ in \eqref{eq:univeuclid}. From \eqref{eq:multieuclid}, it is clear that $[g] = [r]$. However, we do \emph{not} have $r \in B$ in general. Hence, unfortunately, sending $g$ to its remainder $r$ is usually not a normal form\ldots but it is when $f_1, \ldots, f_s$ is a Gr\"obner basis!

A set of polynomials $g_1, \ldots, g_k \in I$ forms a \emph{Gr\"obner basis} of the ideal $I$ if the leading terms ${\rm LT}(g_1), \ldots, {\rm LT}(g_k)$ generate the \emph{leading term ideal} $\langle {\rm LT}(g) \, : \, g \in I \rangle$. We point out that no finiteness of $V_K(I)$ or radicality of $I$ is required for this definition. The remainder $r$ in $g = q_1 \cdot g_1 + \cdots + q_k \cdot g_k + r$ where ${\rm LT}(f_i)$ does not divide any term of $r$, for all $i$, now satisfies $[g] = [r]$ \emph{and $r \in B$}. This justifies the following~claim:
\begin{center}
\emph{Taking remainder upon Euclidean division by a Gr\"obner basis is a normal form}. 
\end{center}
Computing a Gr\"obner basis $g_1, \ldots, g_k$ from a set of input polynomials $f_1, \ldots, f_s$ can be interpreted as Gaussian elimination on a Macaulay matrix \cite{faugere1999new}. Once this has been done, multiplication matrices are computed via taking remainder upon Euclidean division by $\{g_1, \ldots, g_k \}$. 

On a sidenote, we point out that Gr\"obner bases are often used for the \emph{elimination of variables}. For instance, if $g_1, \ldots, g_k$ form a Gr\"obner basis of an ideal $I$ with respect to a \emph{lex} monomial order for which $x_1 \prec x_2 \prec \cdots \prec x_n$, we have for $j = 1, \ldots n$ that the \emph{$j$-th elimination ideal}
\[ I_j \, = \, I \cap K[x_1, \ldots, x_j]  \, = \,  \langle g_i \, :\, g_i \in K[x_1, \ldots, x_j] \rangle \]
is generated by those elements of our Gr\"obner basis which involve only the first $j$ variables, see \cite[Chapter 3, \S 1, Theorem 2]{cox2013ideals}. In our case, a consequence is that one of the $g_i$ is univariate in $x_1$, and its roots are the $x_1$-coordinates of $z_1, \ldots, z_\delta$. The geometric counterpart of computing the $j$-th elimination ideal is \emph{projection} onto a $j$-dimensional coordinate space: the variety $V_K(I_j) \subset K^j$ is obtained from $V_K(I) \subset K^n$ by forgetting the final $n-j$ coordinates $(x_1, \ldots, x_n) \mapsto (x_1, \ldots, x_j)$ and taking the closure of the image. Here are two examples. 
\begin{example}{Example: the projection of a space curve}
\begin{minipage}{0.6 \linewidth}
To the right we show a blue curve in $\mathbb{R}^3$ defined by an ideal $I \subset \mathbb{R}[x,y,z]$. Its Gr\"obner basis with respect to the lex ordering $x \prec y \prec z$ contains $g_1 \in \mathbb{R}[x,y]$, which generates $I_2$. The variety $V_{\mathbb{R}}(I_2) =V_{\mathbb{R}}(g_1)  \subset \mathbb{R}^2$ is the orange curve in the picture.
\end{minipage}
\quad
\begin{minipage}{0.35 \linewidth}
\centering
\includegraphics[height = 3.5cm]{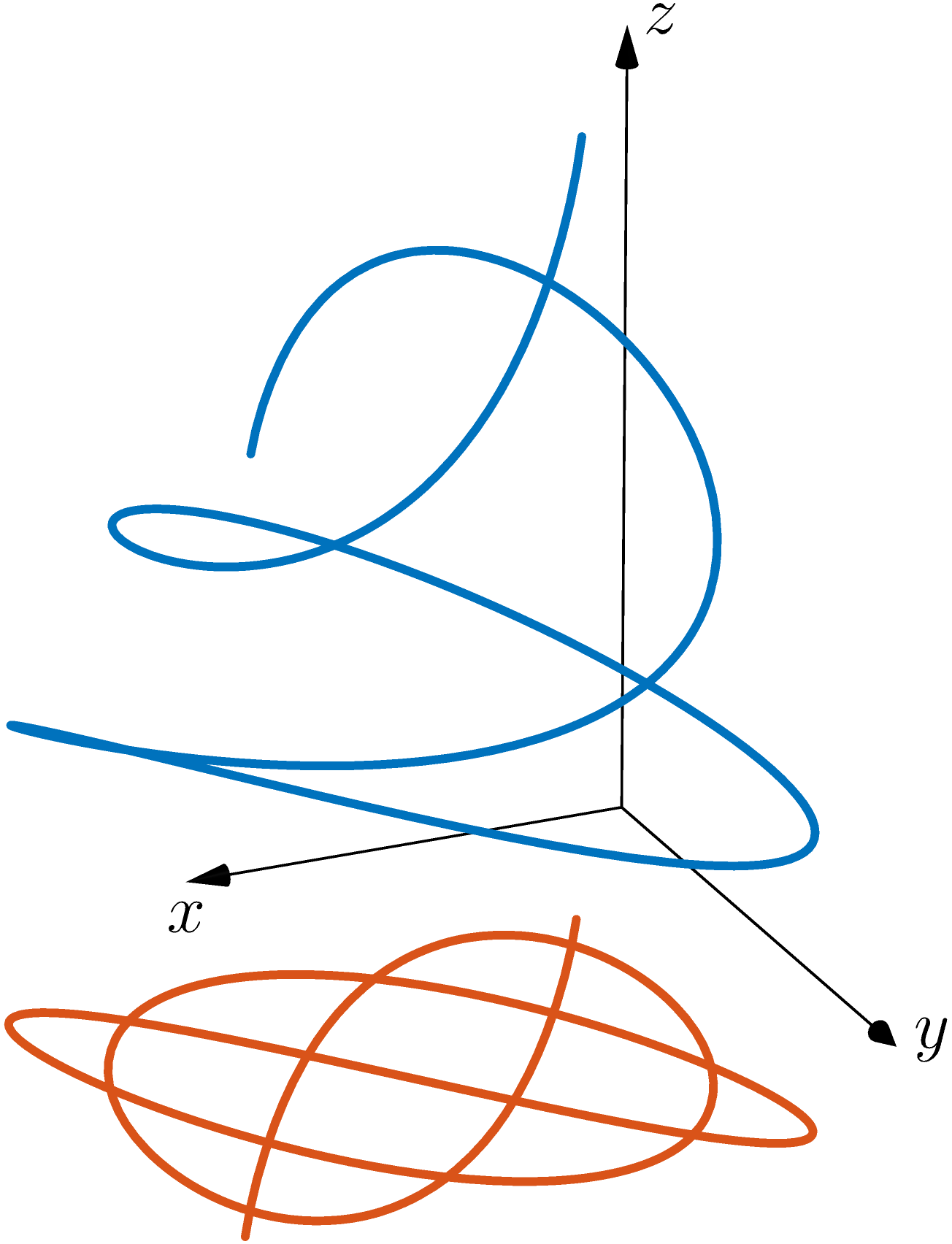}
\vspace{0.2cm}
\end{minipage}
\end{example}

\begin{example}{Example: smooth del Pezzo surfaces}
In \cite{mitankin2020rational}, the authors study del Pezzo surfaces of degree 4 in $\mathbb{P}^4$ with defining equations $x_0x_1-x_2x_3 = a_0x_0^2+a_1x_1^2+a_2x_2^2+a_3x_3^2+a_4x_4^2 = 0$. We will substitute $x_4 = 1 - x_0 - x_1 - x_2 - x_3$ to reduce to the affine case.
It is claimed that the smooth del Pezzo surfaces of this form are those for which the parameters $a_0, \ldots, a_4$ lie outside the hypersurface $H = \{ a_0a_1a_2a_3a_4(a_0a_1-a_2a_3) = 0 \}$. This hypersurface is the projection of the variety 
\[ \left \{ (a, x) \in \mathbb{Q}^5 \times \mathbb{Q}^4 \, : \,  x_0x_1-x_2x_3 = \sum_{i=0}^3 a_i x_i^2 + a_4 \big (1-\sum_{i=0}^3x_i \big )^2 = 0 \text{ and } {\rm rank}(J) < 2  \right \} \]
onto $\mathbb{Q}^5$. Here $J$ is the $2 \times 4$ Jacobian matrix of our two equations with respect to the four variables $x_0, x_1, x_2, x_3$. The defining equation of $H$ is computed in \texttt{Macaulay2} \cite{M2} as follows:

\footnotesize
\begin{verbatim}
R = QQ[x_0..x_3,a_0..a_4]
x_4 = 1-x_0-x_1-x_2-x_3
I = ideal( x_0*x_1-x_2*x_3 , a_0*x_0^2 + a_1*x_1^2 + ... + a_4*x_4^2 )
M = submatrix( transpose jacobian I , 0..3 )
radical eliminate( I+minors(2,M) , {x_0,x_1,x_2,x_3} )
\end{verbatim}
\normalsize
The work behind the final command is a Gr\"obner basis computation. 
\end{example}

\begin{remark}
In a numerical setting, it is better to use \emph{border bases} or more general bases to avoid amplifying rounding errors. Border bases use basis elements $b_i$ for $B$ whose elements satisfy a connectedness property. See, for instance, \cite{mourrain1999new} for details. They do not depend on a monomial order. For a summary and comparison between Gr\"obner bases and border bases, see \cite[Sections 3.3.1, 3.3.2]{telen2020thesis}. Nowadays, bases are selected adaptively by numerical linear algebra routines, such as QR decomposition with optimal column pivoting or singular value decomposition. This often yields a significant improvement in terms of accuracy. See, for instance, Section 7.2 in \cite{telen2018stabilized}. 
\end{remark}
\end{important}

\subsection{Homotopy Continuation} \label{subsec:homotopy}
The goal of this subsection is to briefly introduce the method of homotopy continuation for solving polynomial systems. For more details, we refer to the textbook~\cite{wampler2005numerical}.

We set $K = \mathbb{C}$ and $n = s$. We think of $F = (f_1, \ldots, f_n) \in R$ as an element of a family ${\cal F}$ of polynomial systems. The reader can replace ${\cal F}$ with any of the families seen in Section \ref{sec:nosolutions}. A \emph{homotopy} in ${\cal F}$ with \emph{target system} $F \in {\cal F}$ and \emph{start system} $G \in {\cal F}$  is a continuous deformation of the map $G = (g_1, \ldots, g_n) : \mathbb{C}^n \rightarrow \mathbb{C}^n$ into $F$, in such a way that all systems obtained throughout the deformation are contained in ${\cal F}$. For instance, When $F \in {\cal F}(d_1,\ldots,d_n)$ as in Section \ref{subsec:bez} and $G$ is any other system in ${\cal F}(d_1,\ldots,d_n)$, a homotopy is $H(x;t) = t \cdot F(x) + (1-t) \cdot G(x)$, where $t$ runs from 0 to 1. Indeed, for any fixed $t^* \in [0,1]$, the degrees of the equations remain bounded by $(d_1, \ldots, d_n)$, hence $H(x;t^*) \in {\cal F}(d_1, \ldots, d_n)$. 

The method of homotopy continuation for solving the target system $f_1 = \cdots = f_n = 0$ assumes that a start system $g_1 = \cdots = g_n = 0$ can easily be solved. The idea is that transforming $G$ continuously into $F$ via a homotopy $H(x;t)$ in ${\cal F}$ transforms the solutions of $G$ continuously into those of $F$. Here is an example with $n = 1$.
\begin{example}{Example: $n = s = 1, {\cal F} = {\cal F}(3)$}
Let $f = -6 + 11x - 6x^2 + x^3 = (x-1)(x-2)(x-3)$ be the \emph{Wilkinson polynomial} of degree 3. We view $f = 0$ as a member of ${\cal F}(3)$ and choose the start system $g = x^3-1 = 0$. The solutions of $g = 0$, are the third roots of unity. The solutions of $H(x;t) = \gamma \cdot t \cdot f(x) + (1-t) \cdot g(x)$ travel from these roots to the integers $1, 2, 3$ as $t$ moves from 0 to 1. This is illustrated in Fig.~\ref{fig:wilkinson}. The random complex constant $\gamma$ is needed to \emph{avoid the discriminant}, see below. This is known as the \emph{gamma trick}.
\begin{figure}
    \centering
    \includegraphics[height = 5.5cm]{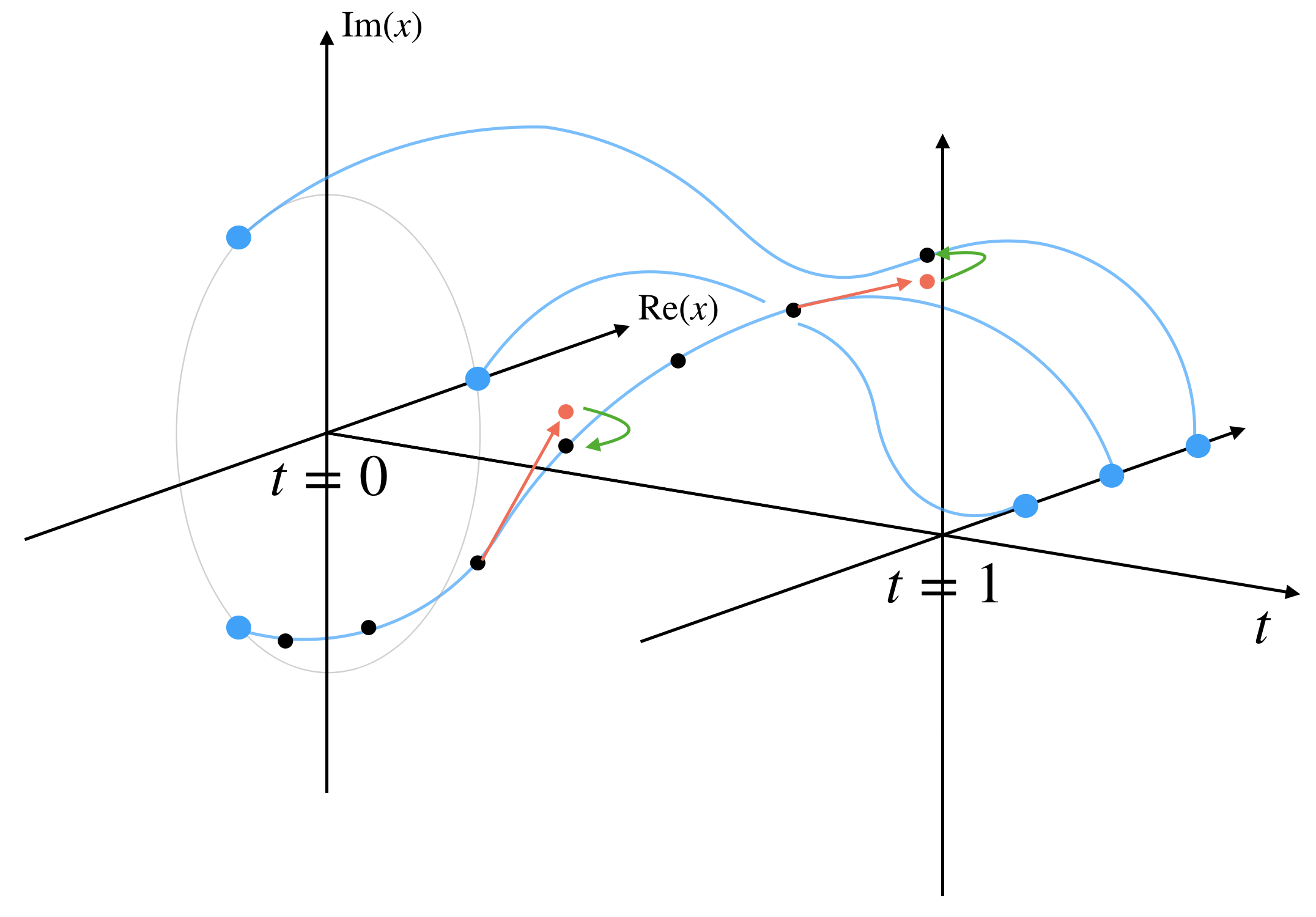}
    \caption{The third roots of unity travel to $1, 2, 3$ along continuous paths.}
    \label{fig:wilkinson}
\end{figure}
\end{example}
More formally, if $H(x;t) = (h_1(x;t), \ldots, h_n(x;t))$ is a homotopy with $t \in [0,1]$, the solutions describe continuous paths $x(t)$ satisfying $H(x(t);t) = 0$. Taking the derivative with respect to $t$ gives the \emph{Davidenko differential equation}
\begin{equation} \label{eq:davidenko}
    \frac{{\rm d} H(x(t),t)}{{\rm d} t} \, = \,  J_x \cdot \Dot{x}(t) + \frac{\partial H}{\partial t}(x(t),t) \, = \, 0, \quad \text{with } J_x = \left ( \frac{\partial h_i}{\partial x_j} \right )_{i,j}. 
\end{equation}
Each start solution $x^*$ of $g_1 = \cdots = g_n = 0$ gives an initial value problem with $x(0) = x^*$, and the corresponding solution path $x(t)$ can be approximated using any numerical ODE method. This leads to a discretization of the solution path, see the black dots in Fig.~\ref{fig:wilkinson}. The solutions of $f_1 = \cdots = f_n = 0$ are obtained by evaluating the solution paths at $t = 1$. The following are important practical remarks. 

\begin{important}{Predict and correct}
Naively applying ODE methods for solving the Davidenko equation \eqref{eq:davidenko} is not the best we can do. Indeed, we have the extra information that the solution paths $x(t)$ satisfy the implicit equation $H(x(t),t) = 0$. This is used to improve the accuracy of the ODE solver in each step. Given an approximation of $x(t^*)$ at any fixed $t^* \in [0,1)$ and a step size $0 < \Delta t \ll 1$, one approximates $x(t^* + \Delta t)$ by $\tilde{x}$ using, for instance, Euler's method. This is called the \emph{predictor} step. Then, one refines $\tilde{x}$ to a satisfactory approximation of $x(t^* + \Delta t)$ by using $\tilde{x}$ as a starting point for Newton iteration on $H(x,t^* + \Delta t) = 0$. This is the \emph{corrector step}. The two-step process is illustrated in Fig.~\ref{fig:wilkinson} (\emph{predict} in orange, \emph{correct} in green, solution paths $x(t)$ in blue). In the right part of the figure, the predict-correct procedure fails: because of a too-large step size $\Delta t$ in the predictor step, the Newton correction converges to a \emph{different} path. This phenomenon is called \emph{path jumping}, and to avoid it one must choose the stepsize $\Delta t$ adaptively. Recent work in this direction uses Pad\'e approximants, see \cite{telen2020robust,timme2021mixed}.
\end{important}

\newpage
\begin{important}{Avoid the discriminant}
Each of the families seen in Section \ref{sec:nosolutions} has a subvariety $\nabla_{\cal F} \subset {\cal F}$ consisting of systems with non-generic behavior in terms of their number of solutions. This subvariety is sometimes referred to as the \emph{discriminant} of ${\cal F}$, for reasons alluded to at the beginning of Section \ref{sec:nosolutions}. When the homotopy $H(x;t)$ crosses $\nabla_{\cal F}$, i.e.~$H(x;t^*) \in \nabla_{\cal F}$ for some $t^* \in [0,1)$, two or more solution paths collide at $t^*$, or some solution paths diverge. This is not allowed for the numerical solution of \eqref{eq:davidenko}. Fortunately, crossing $\nabla_{\cal F}$ can be avoided. The discriminant $\nabla_{\cal F}$ has complex codimension at least one, hence \emph{real} codimension at least two. Since the homotopy $t \mapsto H(x;t)$ describes a one-real-dimensional path in ${\cal F}$, it is always possible to go \emph{around} the discriminant. See for instance \cite[Section 7]{wampler2005numerical}. When the target system $F$ belongs to the discriminant, \emph{end games} are used to deal with colliding/diverging paths at $t = 1$ \cite[Section 10]{wampler2005numerical}. This story implies that the number of paths tracked in a homotopy algorithm is the generic number of solutions of the family ${\cal F}$. In that sense, results like Theorems \ref{thm:bezout}, \ref{thm:kushnirenko} and \ref{thm:BKK} characterize the complexity of homotopy continuation in the respective families.
\end{important}

\begin{important}{Start systems in practice} There are recipes for start systems in the families ${\cal F}$ from Section \ref{sec:nosolutions}. For instance, we use $G = (x_1^{d_1}-1, \ldots, x_n^{d_n}-1)$ for ${\cal F}(d_1, \ldots, d_n)$. The $d_1 \cdots d_n$ solutions can easily be written down. Note that $G \notin \nabla_{d_1, \ldots, d_n}$. For the families ${\cal F}({\cal A})$ and ${\cal F}({\cal A}_1, \ldots, {\cal A}_n)$, an algorithm to solve start systems was developed in \cite{huber1995polyhedral}. For solving start systems of other families, one may use \emph{monodromy loops} \cite{duff2019solving}.
\end{important}

\section{Case study: 27 lines on the Clebsch surface} \label{sec:lines}
A classical result from intersection theory states that every smooth cubic surface in complex three-space contains exactly 27 lines. In this final section, we use Gr\"obner bases and homotopy continuation to compute lines on the Clebsch surface defined by \eqref{eq:fclebsch}. This particular surface is famous for the fact that all its 27 lines are real. Let $f(x,y,z)$ be as in \eqref{eq:fclebsch}. A line in $\mathbb{R}^3$ parameterized by $(a_1 + t \cdot b_1, a_2 + t \cdot b_2, a_3 + t \cdot b_3)$ is contained in our Clebsch surface if and only if 
$f (  a_1 + t \cdot b_1, a_2 + t \cdot b_2, a_3 + t \cdot b_3  ) \equiv  0 .$
The left-hand side evaluates to a cubic polynomial in $t$ with coefficients in the ring $\mathbb{Z}[a_1,a_2,a_3,b_1,b_2,b_3] = \mathbb{Z}[a,b]$:
\[ f\, ( \, a_1 + t \cdot b_1, a_2 + t \cdot b_2, a_3 + t \cdot b_3 \, ) \, = \, f_1(a,b) \cdot t^3 + f_2(a,b) \cdot t^2 + f_3(a,b) \cdot t + f_4(a,b).\]
The lines contained in the Clebsch surface satisfy
\begin{equation} \label{eq:4eqs}
    f_1(a,b) \, = \, f_2(a,b) \, = \, f_3(a,b) \, = \, f_4(a,b) \, = \, 0. 
\end{equation} 
We further reduce this to a system of $s = 4$ equations in $n = 4$ unknowns by removing the redundancy in our parameterization of the line: the space of lines in three-space (i.e.~the Grassmannian $\mathbb{G}(1,3)$) has dimension four, not six. We may impose a random affine-linear relation among the $a_i$ and $b_i$. We choose to substitute 
\[ a_3 = -(7+a_1+3 a_2)/5, \quad b_3 = -(11+3b_1+5b_2)/7. \]

Implementations of Gr\"obner bases are available, for instance, in {\tt Maple} \cite{maple} and in ${\tt msolve}$ \cite{berthomieu2021msolve}, which can be used in ${\tt julia}$ via the package ${\tt msolve.jl}$. This is also available in the package \texttt{Oscar.jl} \cite{OSCAR}. The following snippet of \texttt{Maple} code constructs our system \eqref{eq:4eqs} and computes a Gr\"obner basis with respect to the graded lexicographic monomial ordering with $a_1 \succ a_2 \succ b_1 \succ b_2$. This basis consists of 23 polynomials $g_1, \ldots, g_{23}$.

\scriptsize
\begin{verbatim}
> f := 81*(x^3 + y^3 + z^3) - 189*(x^2*y + x^2*z + x*y^2 + x*z^2 + y^2*z + y*z^2) 
       + 54*x*y*z + 126*(x*y + x*z + y*z) - 9*(x^2 + y^2 + z^2) - 9*(x + y + z) + 1:
> f := expand(subs({x = t*b[1] + a[1], y = t*b[2] + a[2], z = t*b[3] + a[3]}, f)):
> f := subs({a[3] = -(7 + a[1] + 3*a[2])/5, b[3] = -(11 + 3*b[1] + 5*b[2])/7}, f):
> ff := coeffs(f, t):
> with(Groebner):
> GB := Basis({ff}, grlex(a[1], a[2], b[1], b[2]));
> nops(GB);                                                ----> output: 23
\end{verbatim}
\normalsize

\noindent The set of standard monomials is the first output of the command \texttt{NormalSet}. It consists of 27 elements, and the multiplication matrix with respect to $a_1$ in this basis is constructed using \texttt{MultiplicationMatrix}:

\scriptsize
\begin{verbatim}
> ns, rv := NormalSet(GB, grlex(a[1], a[2], b[1], b[2])):
> nops(ns);                                                ----> output: 27
> Ma1 := MultiplicationMatrix(a[1], ns, rv, GB, grlex(a[1], a[2], b[1], b[2])):
\end{verbatim}
\normalsize

\noindent This is a matrix of size $27 \times 27$ whose eigenvectors reveal the solutions (Theorem~\ref{thm:EVthm}). 
We now turn to \texttt{julia} and use \texttt{msolve} to compute the 27 lines on $\{f = 0 \}$ as~follows: 

\footnotesize
\begin{verbatim}
using Oscar
R,(a1,a2,b1,b2) = PolynomialRing(QQ,["a1","a2","b1","b2"])
I = ideal(R, [-189*b2*b1^2 - 189*b2^2*b1 + 27*(11 + 3*b1 + 5*b2)*b1^2 + ...
A, B = msolve(I)
\end{verbatim}
\normalsize

\noindent The output \texttt{B} contains 4 rational coordinates $(a_1,a_2,b_1,b_2)$ of 27 lines which approximate the solutions. To see them in floating point format, use for instance 

\footnotesize
\begin{verbatim}
[convert.(Float64,convert.(Rational{BigInt},b)) for b in B]
\end{verbatim}
\normalsize

\noindent We have drawn three of these lines on the Clebsch surface in Fig.~\ref{fig:clebsch} as an illustration. Other software systems supporting Gr\"obner bases are \texttt{Macaulay2} \cite{M2}, \texttt{Magma} \cite{MR1484478}, \texttt{Mathematica} \cite{mathematica} and \texttt{Singular} \cite{greuel2001singular}. 

\begin{figure}
    \centering
    \includegraphics[height = 3.6cm]{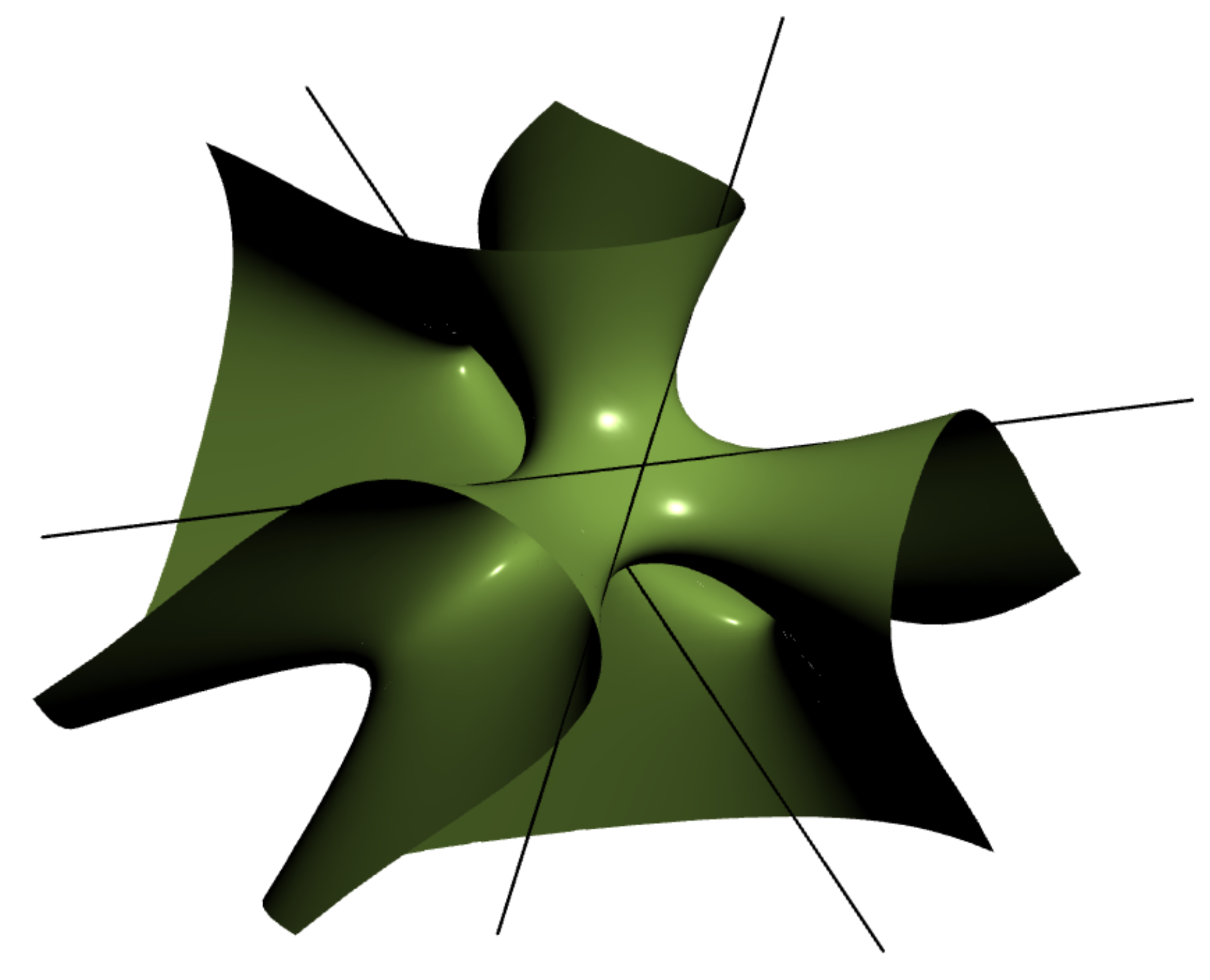}
    \qquad
    \includegraphics[height = 3.6cm]{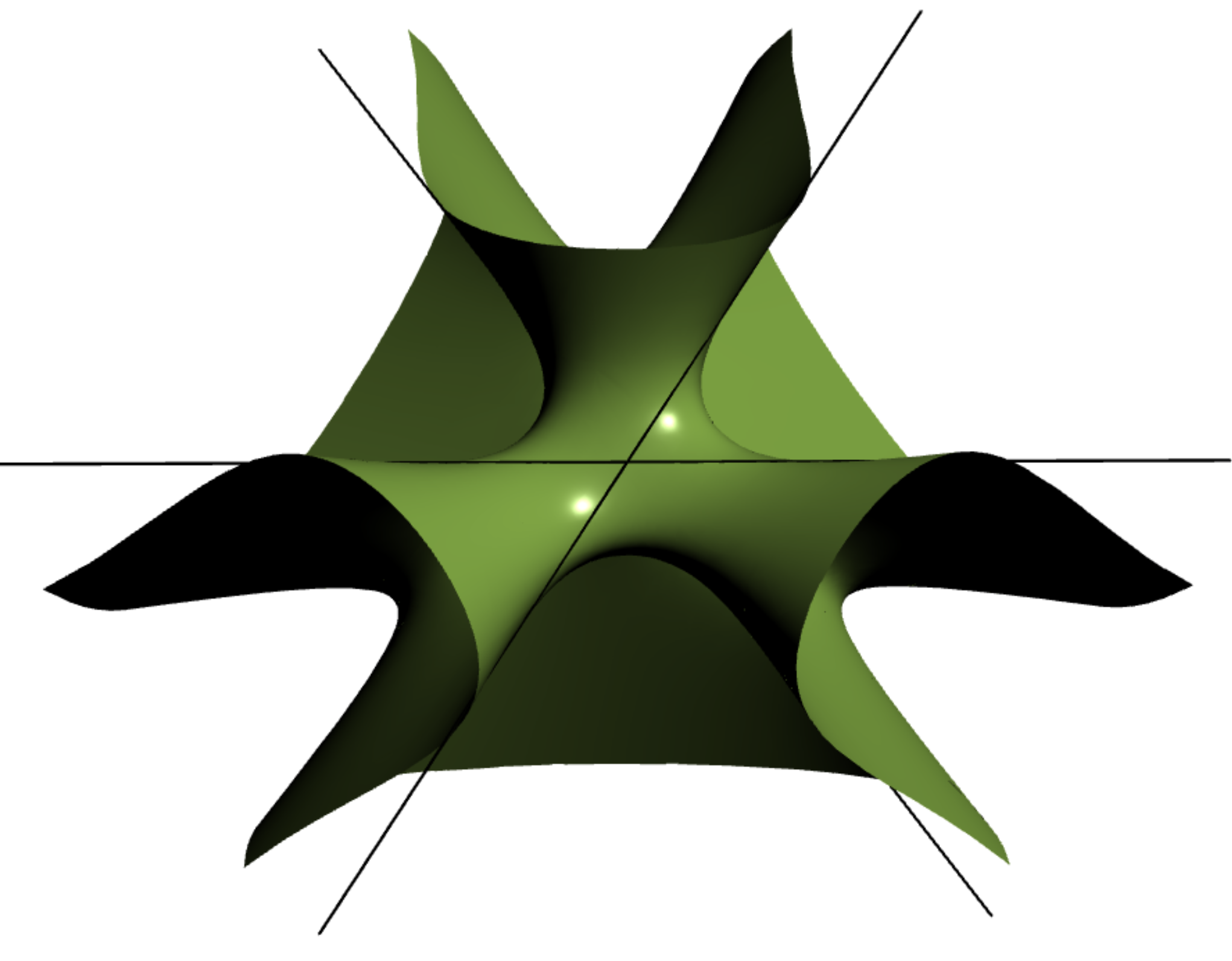}
    \caption{Two views on the Clebsch surface with three of its 27 lines. }
    \label{fig:clebsch}
\end{figure}

Homotopy continuation methods provide an alternative way to compute our 27 lines. Here we use the \texttt{julia} package \texttt{HomotopyContinuation.jl} \cite{breiding2018homotopycontinuation}.

\scriptsize
\begin{verbatim}
using HomotopyContinuation
@var x y z t a[1:3] b[1:3]
f = 81*(x^3 + y^3 + z^3) - 189*(x^2*y + x^2*z + x*y^2 + x*z^2 + y^2*z + y*z^2) 
    + 54*x*y*z + 126*(x*y + x*z + y*z) - 9*(x^2 + y^2 + z^2) - 9*(x + y + z) + 1
fab = subs(f, [x;y;z] => a+t*b)
E, C = exponents_coefficients(fab,[t])
F = subs(C,[a[3];b[3]] => [-(7+a[1]+3*a[2])/5; -(11+3*b[1]+5*b[2])/7])
R = solve(F)
\end{verbatim}
\normalsize

\noindent The output is shown in Fig.~\ref{fig:output}. There are 27 solutions, as expected. The last line indicates that a \texttt{:polyhedral} start system was used. In our terminology, this means that the system was solved using a homotopy in the family ${\cal F}({\cal A}_1, \ldots, {\cal A}_4)$ from Section \ref{subsec:bkk}. The number of tracked paths is 45, which is the mixed volume ${\rm MV}({\cal A}_1, \ldots, {\cal A}_4)$ of this family. The discrepancy $27 < 45$ means that our system $F$ lies in the discriminant $\nabla_{{\cal A}_1, \ldots, {\cal A}_4}$. The 18 `missing' solutions are explained in \cite[Section 3.3]{bender2022toric}. The output also tells us that all solutions have multiplicity one (this is the meaning of \texttt{non-singular}) and all of them are real. 

\begin{figure}[h!]
    \centering
    \includegraphics[height = 1.8cm]{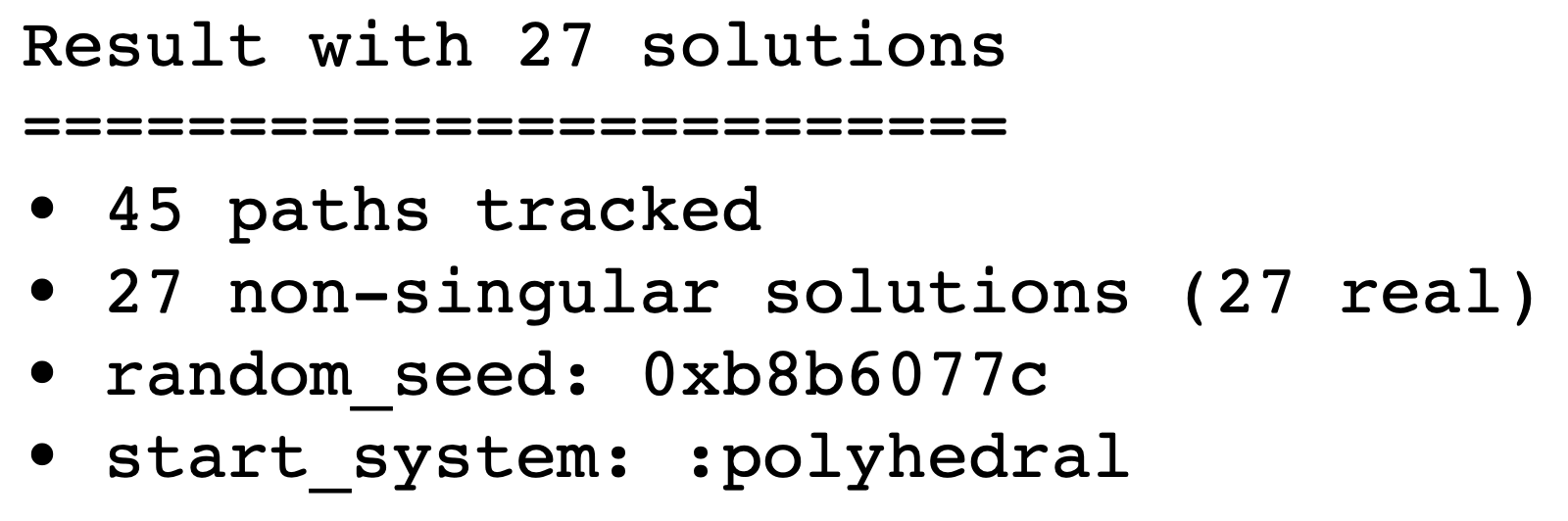}
    \caption{The \texttt{julia} output when computing 27 lines on the Clebsch surface.}
    \label{fig:output}
\end{figure}

\noindent Other software implementing homotopy continuation are \texttt{Bertini} \cite{bates2013numerically} and \texttt{PHCpack} \cite{verschelde1999algorithm}. Numerical normal form methods are used in \texttt{EigenvalueSolver.jl} \cite{bender2021yet}.

\vspace{-0.5cm}

\input{myreferences}
\end{document}

%% file: hexagon.tex
\begin{tikzpicture}[scale = 0.8]
\begin{axis}[%
width=5cm,
height=5cm,
scale only axis,
xmin=-0.5,
xmax=2.5,
xtick = \empty,
ymin=-0.5,
ymax=2.5,
ytick = \empty,
axis background/.style={fill=white},
axis line style={draw=none}
]
\addplot [color=black,solid,thick, fill opacity = 0.2, fill = mycolor4,forget plot]
  table[row sep=crcr]{%
0        1\\
1        0\\
2        0\\
2        1\\
1       2\\
0       2 \\
0        1\\
};
\addplot[only marks,mark=*,mark size=2.0pt,black
        ]  coordinates {
    (1,0) (0,1) (2,1) (1,1) (1,2) (0,2) (2,0)
};
\node at (axis cs:1.20,0.20) {\small $x$};
\node at (axis cs:1.20,1.20) {\small $xy$};
\node at (axis cs:1.20,2.20) {\small $xy^2$};
\node at (axis cs:0.20,1.20) {\small $y$};
\node at (axis cs:0.20,2.20) {\small $y^2$};
\node at (axis cs:2.20,0.20) {\small $x^2$};
\node at (axis cs:2.20,1.20) {\small $x^2y$};
\end{axis}
\end{tikzpicture}%